\documentclass{amsart}
\usepackage{enumerate}
\usepackage{amssymb}
\usepackage{mathrsfs}

\title
{The method of forcing}

\newtheorem{thm}{Theorem}[section]
\newtheorem{lem}[thm]{Lemma}
\newtheorem{prop}[thm]{Proposition}
\newtheorem{cor}[thm]{Corollary}
\newtheorem{obs}[thm]{Observation}
\newtheorem{property}{Property}

\theoremstyle{remark}
\newtheorem{remark}[thm]{Remark}
\newcounter{my_enumerate_counter}
\newcommand{\pushcounter}{\setcounter{my_enumerate_counter}{\value{enumi}}}
\newcommand{\popcounter}{\setcounter{enumi}{\value{my_enumerate_counter}}}

\theoremstyle{definition}
\newtheorem{defn}[thm]{Definition}

\newtheorem{example}[thm]{Example}

\newcommand{\Seq}[1]{( #1 )}
\newcommand{\Th}{{}^{\textrm{th}}}

\newcommand\act{\curvearrowright}
\newcommand\Int[2]{{#1}({#2})}
\newcommand{\forces}{\Vdash}
\newcommand{\lbrak}{[\mspace{-2.2mu}[}
\newcommand{\rbrak}{]\mspace{-2.2mu}]}
\newcommand{\truth}[1]{\lbrak #1 \rbrak}

\newcommand{\one}{\mathbf{1}}

\newcommand{\C}{\mathtt{c}}

\newcommand{\bigmeet}{\bigwedge}
\newcommand{\bigjoin}{\bigvee}
\newcommand\mand{\textrm{ and }}

\newcommand{\dom}{\operatorname{dom}}
\newcommand{\cl}{\operatorname{cl}}
\newcommand\cf{\operatorname{cf}}
\newcommand\cat{{}^\smallfrown}
\newcommand{\symdif}{\triangle}
\newcommand{\BC}{\operatorname{BC}}

\newcommand\Rand{\mathcal{R}}
\newcommand\Cohen{\mathcal{C}}
\newcommand\Mathias{\mathcal{M}}
\newcommand\Ameoba{\mathcal{A}}
\newcommand\Ocal{\mathcal{O}}

\newcommand\R{\mathbf{R}}
\newcommand\N{\mathbf{N}}
\newcommand\Z{\mathbf{Z}}
\newcommand\Q{\mathbf{Q}}

\newcommand\Bcal{\mathcal{B}}
\newcommand\Ccal{\mathcal{C}}
\newcommand\Dcal{\mathcal{D}}
\newcommand\Ecal{\mathcal{E}}
\newcommand\Fcal{\mathcal{F}}

\newcommand\Pcal{\mathcal{P}}
\newcommand\Qcal{\mathcal{Q}}
\newcommand\Rcal{\mathcal{R}}
\newcommand\Scal{\mathcal{S}}
\newcommand\Ucal{\mathcal{U}}
\newcommand\Xcal{\mathcal{X}}

\keywords{forcing, generically absolute, random, universally Baire}

\subjclass[2010]{Primary: 03E40; Secondary: 03E57, 03E75, 05D10, 54C35}

\author{Justin Tatch Moore}

\address{Department of Mathematics \\ Cornell University\\
Ithaca, NY 14853--4201 \\ USA}

\thanks{The present article started out as a set of notes prepared a
tutorial presented during
workshop 13w5026 at the Banff International Research Station in November 2013.
The author would like to thank BIRS for its generous
hospitality during the meeting.
The author's research and travel to the workshop
was supported in part by NSF grant DMS-1262019;
the preparation of this article was supported
in part by NSF grant DMS-1600635.}

\email{{\tt justin@math.cornell.edu}}
 
\begin{document}

\begin{abstract}
The purpose of this article is to give a presentation of the method
of forcing aimed at someone with a minimal knowledge of set theory and logic.
The emphasis will be on how the method can be used to prove theorems
in ZFC.
\end{abstract}

\maketitle

\section{Introduction}

Let us begin with two thought experiments.
First consider the following ``paradox'' in probability:
if $Z$ is a continuous random variable, then for any possible outcome
$z$ in $\R$, the event $Z \ne z$ occurs \emph{almost surely} (i.e. with probability $1$).
How does one reconcile this with the fact that, in a truly random outcome, every event having probability 1 should occur?
Recasting this in more formal language we have that,
``for all $z \in \R$, \emph{almost surely} $Z \ne z$'',
while ``\emph{almost surely} there exists a $z \in \R$, $Z = z$.''

Next suppose that, for some index set $I$, $\Seq{Z_i : i \in I}$ is a family of independent continuous random variables.
It is a trivial matter that for each pair $i \ne j$, the inequality $Z_i \ne Z_j$ holds with probability $1$.
For large index sets $I$, however, 
\[
|\{Z_i : i \in I\}| = |I|
\]
holds with probability $0$; in fact this event contains no outcomes if $I$ is larger in cardinality than $\R$.
In terms of the formal logic, we have that,
``for all $i\ne j$ in $I$, \emph{almost surely} the event $Z_i \ne Z_j$ occurs'', while
``\emph{almost surely it is false that} for all $i \ne j \in I$, the event $Z_i \ne Z_j$ occurs''.

It is natural to ask whether it is possible to revise the notion of \emph{almost surely}
so that its meaning remains unchanged for simple logical assertions such as $Z_i \ne Z_j$
but such that it commutes with quantification.
For instance one might reasonably ask that, in the second example, 
$|\{Z_i : i \in I\}| = |I|$ should occur \emph{almost surely} regardless of the cardinality of the index set.
Such a formalism would describe truth in a necessarily larger model of mathematics,
one in which there are new outcomes to the random experiment which did not exist before the experiment was
performed.

The method of forcing, which was introduced by Paul Cohen to establish the independence of the Continuum
Hypothesis \cite{CON_negCH} and put into its current form by Scott \cite{ind_CH:Scott} and Solovay \cite{solovay_model},
addresses issues precisely of this kind.
From a modern perspective, forcing provides a formalism for examining what occurs \emph{almost surely}
not only in probability spaces but also in a much more general setting than what is provided by
our conventional notion of randomness.
Forcing has proved extremely useful in developing and understanding of models of set theory and in determining what
can and cannot be proved within the standard axiomatization of mathematics (which we will take to be ZFC).
In fact it is a heuristic of modern set theory that if a statement arises naturally in mathematics and is consistent, then
its consistency can be established using forcing,
possibly starting from a large cardinal hypothesis.

The focus of this article, however, is of a different nature:
the aim is to demonstrate how the method of forcing can be used to \emph{prove theorems} as opposed to
\emph{establish consistency results}.
Forcing itself concerns the study of adding \emph{generic objects} to a model of set theory, resulting in a larger model of
set theory.
One of the key aspects of forcing is that it provides a formalism for studying what happens \emph{almost surely}
as the result of introducing a generic object.
An analysis of this formalism sometimes leads to new results concerning the original model itself --- results which are in fact independent
of the model entirely.
This can be likened to how the probabilistic method is used in finite combinatorics in settings where more constructive methods fail (see, e.g., \cite{probabilistic_method}).

In what follows, we will examine several examples of how forcing can be used to prove theorems.
Admittedly there are relatively few examples of such applications thus far.
It is my hope that by reaching out to a broader audience, this article will inspire more
innovative uses of forcing in the future.

Even though the goals and examples are somewhat unconventional,
the forcings themselves and the development of the theory
are much the same as one would encounter in a more standard treatment of forcing.
The article will only assume a minimal knowledge of set theory and logic, similar to what a graduate or advanced undergraduate
student might encounter in their core curriculum.
In particular, no prior exposure to forcing will be assumed.

The topics which will be covered include the following:
\emph{genericity},
\emph{names},
the \emph{forcing relation},
\emph{absoluteness},
the \emph{countable chain condition},
\emph{countable closure}, and
\emph{homogeneity arguments}.
These concepts will be put to use though several case studies:
\begin{enumerate} 

\item partition theorems of Galvin, Nash-Williams, and Prikry for finite and infinite subsets of $\omega$;

\item intersection properties of uncountable families of events in a probability space;

\item a partition theorem of Halpern and L\"auchli for products of finitely branching trees;

\item a property of marker sequences for Bernoulli shifts;

\item Todorcevic's analysis of Rosenthal compacta.

\pushcounter
\end{enumerate}
Sections marked with a `*' will not be needed for later sections.

While we will generally avoid proving consistency results,
the temptation to establish the consistency
of the Continuum Hypothesis and its negation along the way is too great ---
this will be for free given what needs to be developed.
For those interested in supplementing the material in this article with a more conventional approach to forcing,
Kunen's \cite{set_theory:Kunen} is widely considered to be the standard treatment.
It also provides a complete introduction to combinatorial set theory and independence results;
the reader looking for further background on set theory is referred there.
See also \cite{multiple_forcing}, \cite{ind_CH:Scott}, \cite{forcing:Shoenfield}, \cite{solovay_model}, \cite{forcing_appl}.
The last section contains a list of additional suggestions for further reading.

This exposition grew out of a set of lecture notes prepared for workshop 13w5026
on ``Axiomatic approaches to forcing techniques in set theory''
at the Banff International Research Station in November 2013.
None of the results presented below are my own.
I'll finish by saying that this project was inspired by countless
conversations with Stevo Todorcevic over the years, starting with my time as his student
in the 1990s.

\section{Preliminaries}

\label{prelim:sec}

Before beginning, we will fix some conventions and
review some of the basic concepts from set theory which will be needed.
Further reading and background can be found in \cite{set_theory:Kunen}.
A set $x$ is \emph{transitive} if whenever $z \in y$ and $y \in x$, then $z \in x$.
Equivalently, $x$ is transitive if and only if every element of $x$ is also a subset of $x$.
An \emph{ordinal} is a set $\alpha$ which is transitive and wellordered by $\in$.
It is readily checked that every element of an ordinal is also an ordinal.
Every wellorder is isomorphic to an ordinal; moreover this ordinal and the isomorphism are unique.
If $\alpha$ and $\beta$ are two ordinals, then exactly one of the following is true:
$\alpha \in \beta$, $\beta \in \alpha$, or $\alpha = \beta$.
We will often write $\alpha < \beta$ to denote $\alpha \in \beta$ if $\alpha$ and $\beta$ are ordinals.

Notice that an ordinal is the set of ordinals smaller than it.
The least ordinal is the emptyset, which is denoted $0$.
If $\alpha$ is an ordinal, then $\alpha + 1$ is the least ordinal greater than $\alpha$;
this coincides with the set $\alpha \cup \{\alpha\}$.
The finite ordinals therefore coincide with the nonnegative integers:
$n := \{0,\ldots,n-1\}$.
The least infinite ordinal is $\omega := \{0,1,\ldots\}$, which coincides with the set of 
natural numbers.
We will adopt the convention that the set $\N$ of natural numbers does not include $0$.
Unless otherwise specified, $i,j,k,l,m,n$ will be used to denote finite ordinals.

An ordinal $\kappa$ is a \emph{cardinal} if whenever $\alpha < \kappa$, $|\alpha| < |\kappa|$.
If $\alpha$ is an ordinal which is not a successor, then we say that $\alpha$ is a limit ordinal.
In this case, the \emph{cofinality} of $\alpha$ is the minimum cardinality of a cofinal subset of $\alpha$.
A cardinal $\kappa$ is \emph{regular} if its cofinality is $\kappa$.
The regularity of a cardinal $\kappa$ is equivalent to the assertion that if $\kappa$ is 
partitioned into fewer than $\kappa$ sets, then one of these sets has cardinality $\kappa$.
If $\kappa$ is a cardinal, then $\kappa^+$ denotes the least cardinal greater than $\kappa$.
Cardinals of the form $\kappa^+$ are called \emph{successor cardinals} and are always regular.

Since every set can be wellordered, every set has the same cardinality as some (unique) cardinal;
we will adopt the convention that $|x|$ is the cardinal $\kappa$ such that $|x| = |\kappa|$.
If $\alpha$ is an ordinal, then we define $\omega_\alpha$ to be the $\alpha\Th$ infinite cardinal.
Thus $\omega_0 := \omega$ and $\omega_\beta := \sup_{\alpha < \beta} (\omega_\alpha)^+$ if $\beta > 0$.
The Greek letters $\alpha$, $\beta$, $\gamma$, $\xi$, and $\eta$ will be used to refer to ordinals;
the letters $\kappa$, $\lambda$, $\mu$, and $\theta$ will be reserved for cardinals.

If $A$ and $B$ are sets, then $B^A$ will be used to denote the collection of all functions from $A$ into $B$.
For us, a function is simply a set of ordered pairs.
Thus if $B \subseteq C$, then $B^A \subseteq C^A$ and if $f$ and $g$ are functions, $f \subseteq g$ means
that $f$ is a restriction of $g$.
There is one exception to this notation worth noting.
We will follow the custom of writing $\aleph_\alpha$ for $\omega_\alpha$ in situations where the underlying
order structure is unimportant (formally $\aleph_\alpha$ equals $\omega_\alpha$).
Arithmetic expressions involving $\aleph_\alpha$'s will be used to refer to the cardinality of the resulting set.
For instance $2^{\omega_1}$ is a collection of functions whose cardinality is $2^{\aleph_1}$, which is a cardinal.


If $(X,d)$ is a metric space, then we define its \emph{completion}
to be the set of equivalence classes of Cauchy sequences,
where $\Seq{x_n : n < \infty}$ is equivalent to $\Seq{y_n : n < \infty}$ if $d(x_n,y_n) \to 0$.
In particular, we will regard the set of real numbers $\R$
as being the completion of the rational numbers $\Q$
with the usual metric $d(p,q) := |p-q|$.
Notice that even if $X$ is complete, the completion is not literally equal to $X$, even though it is canonically isometric
to $X$.
This will serve as a minor annoyance when we define names for complete metric spaces in Section \ref{semantics:sec}.

Finally, we will need some notation from first order logic.
The \emph{language of set theory} is the first order
language with a single binary relation $\in$.
If $\phi$ is a formula in the language of set theory, then $\phi(v_1,\ldots,v_n)$ will
be used to indicate that every free variable in $\phi$ is
$v_i$ for some $i =1,\ldots, n$. 
If $x_1, \ldots, x_n$ are constants, then $\phi(x_1,\ldots,x_n)$ is the result of simultaneously substituting $x_i$ for $v_i$ for each $i$.
If $\phi$ is a formula and $v$ is variable and $x$ is a term, then $\phi[x/v]$ is the result of substituting
$x$ for every free occurrence of $v$ in $\phi$ (of which there may be none). 
If $\phi$ has no free variables, then we say that $\phi$ is a \emph{sentence}.
If every quantifier in $\phi$ is 
of the form $\exists x \in y$ or $\forall x \in y$ for some variables $x$ and $y$,
then we say that the quantification in $\phi$ is \emph{bounded}.
Many assertions can be expressed using only bounded quantification:
for instance the assertions ``$A = \bigcup B$'' and ``$(\Qcal,\leq)$ is a partially ordered set''
are expressible by formulas with only bounded quantification.

We now recall some foundational results in set theory which justify our emphasis
on transitive models of set theory below.
A binary relation $R$ is \emph{well founded} if there is no infinite sequence
$\Seq{x_n : n < \infty}$ such that $x_{n+1} R x_n$.
A binary relation $R$ on a set $X$ is \emph{extensional} if for all $x$ and $y$ in $X$,
$\{z \in X : z R x\} = \{z \in X : z R y\}$ implies $x = y$.
Among the axioms of ZFC are the assertions that $\in$ is well founded and extensional.

\begin{prop}
Suppose that $(X,E)$ is a binary relation and $E$ is well founded and extensional.
Then $(X,E)$ is uniquely isomorphic to a transitive set equipped with $\in$.
In particular, if $(X,E)$ is a model of ZFC and $E$ is well founded, then
$(X,E) \simeq (M,\in)$ for some transitive set $M$.
\end{prop}

\begin{prop}
If $M$ is a transitive set, $\phi(v_1,\ldots,v_n)$ is a first order formula with only bounded
quantification and $a_1,\ldots,a_n \in M$,
then $(M,\in) \models \phi(a_1,\ldots,a_n)$ if and only if $\phi(a_1,\ldots,a_n)$ is true.
\end{prop}

Thus, for example, if $M$ is a transitive set and $\Qcal$ is a partial order in $M$,
then $(M,\in)$ satisfies that $\Qcal$ is a partial order.

\section{What is forcing?}

\label{what_is:sec}
Forcing is the procedure of adjoining to a model $M$ of set theory a new \emph{generic object} $G$ in order to
create a larger model $M[G]$.
In this context, 
$M$ is referred to as the \emph{ground model} and $M[G]$ is a \emph{generic extension} of $M$.
For us, the generic object will always be a new subset of some partially ordered set $\Qcal$ in $M$,
known as a \emph{forcing}.
This procedure has the following desirable properties:
\begin{enumerate} \popcounter

\item $M[G]$ is also a model of set theory and is the minimal model of set theory which has as members all the elements of $M$
and also the generic object $G$.

\item The truth of mathematical statements in $M[G]$ can be determined by a formalism within $M$, known as
the \emph{forcing relation}, which is completely specified by $\Qcal$.
The workings of this formalism are purely internal to $M$.

\pushcounter
\end{enumerate}
While it will generally not concern us in this article, the key meta-mathematical feature of forcing is that it is often
the case that it is easier to determine truth in the generic extension $M[G]$ than in the \emph{ground model} $M$.
For instance Cohen specified the description of a forcing $\Qcal$ with the property that if $M[G]$ is any generic extension
created by forcing with $\Qcal$, then $M[G]$ necessarily satisfies that the Continuum Hypothesis is false
\cite{CON_negCH} (see Section \ref{ccc:sec} below).
In fact the second thought experiment presented at the beginning of the introduction is derived from a variation of this forcing.
It is also not difficult to specify different forcings which always produce generic extensions satisfying the Continuum Hypothesis
(see Section \ref{sigma-closed:sec} below).

There are two perspectives one can have of forcing: one which is primarily semantic
and one which is primarily syntactic.
Each has its own advantages and disadvantages.
The semantic approach makes certain properties of the forcing relation and the generic
extension intuitive and transparent.
On the other hand, it is fraught with metamathematical issues and philosophical
hangups.
The syntactic approach is less intuitive but more elementary and makes certain other
features of forcing constructions more transparent.
We will tend to favor the syntactic approach in what follows.
We will now fix some terminology.

\begin{defn}[forcing]
A \emph{forcing} is a set $\Qcal$ equipped with a transitive reflexive relation $\leq_\Qcal$ which
contains a greatest element $\one_\Qcal$.
If $\Qcal$ is clear from the context, the subscripts are usually suppressed.
\end{defn}

Our prototypical example of a forcing is $\Rand$,
the collection of all measurable subsets of $[0,1]$ having positive Lebesgue measure,
ordered by containment.
Elements of a forcing are often referred to as \emph{conditions} and are regarded as being approximations to 
a desired \emph{generic object}.
In the analogy with randomness, the conditions correspond to the events of the probability space which have positive measure.
If $q \leq p$, then we sometime say that $q$ is \emph{stronger than} $p$ or that $q$ \emph{extends} $p$.
We think of $q$ as providing a better approximation to the generic object.
It will be helpful to abstract the notion of an outcome in terms of a collection of mutually compatible events.
A set $G \subseteq \Qcal$ is a \emph{filter} if $G$ is nonempty, \emph{upward closed}, and \emph{downward directed} in $\Qcal$:
\begin{enumerate}
\popcounter

\item if $q$ is in $G$, $p$ is in $\Qcal$ and $q \leq p$, then $p$ is in $G$;

\item if $p$ and $q$ are in $G$, then there is an $r$ in $G$ such that $r \leq p,q$.

\pushcounter
\end{enumerate}
If $p$ and $q$ are in $\Qcal$, then we say that $p$ and $q$ are \emph{compatible}
if there is a $r$ in $\Qcal$ such that $r \leq p,q$.
Otherwise we say that $p$ and $q$ are \emph{incompatible}.
Notice that two conditions are compatible exactly when there is a filter which contains both of them.
Of course two events in $\Rand$ are compatible exactly when they intersect in a set of positive measure.

A forcing $\Qcal$ is \emph{separative} if whenever $p \not \leq q$, there is an $r \leq p$ such
that $r$ and $q$ are incompatible.
Notice that if $\Qcal$ is any forcing, we can define an equivalence relation $\equiv$ on $\Qcal$
by $q \equiv p$ if 
\[
\{r \in \Qcal : r \textrm{ is compatible with } p\} = \{r \in \Qcal : r \textrm{ is compatible with } q\}.
\]
The quotient is ordered by $[q] \leq [p]$ if
\[
\{r \in \Qcal : r \textrm{ is compatible with } q\} \subseteq \{r \in \Qcal : r \textrm{ is compatible with } p\}.
\]
This quotient ordering is separative and is known as the \emph{separative quotient}.
Notice that if $\Qcal$ is separative, then $\equiv$ is just equality and the quotient ordering is just the usual ordering.
The forcing $\Rand$ is not separative; in this example $p \equiv q$ if $p$ and $q$ differ by a measure $0$ set.
It is often convenient to assume forcings are separative and we will often pass to the separative quotient without
further mention (just as one often writes equality of functions in analysis when they
really mean equality modulo a measure 0 set).

The following definition will play a central role in all that follows.

\begin{defn}[generic]
If $M$ is a collection of sets and $\Qcal$ is a forcing, then we say that a filter $G \subseteq \Qcal$ is \emph{$M$-generic}
if whenever $E \subseteq \Qcal$ is in $M$,
there is a $p \in G$ which is either in $E$ or is incompatible with every element of $E$.
\end{defn}

A family $E$ of conditions is said to be \emph{exhaustive}
if whenever $p$ is an element of $\Qcal$, there is an element $q$ of $E$ which is compatible with $p$.
Notice that if $\Ecal$ is a collection of exhaustive sets and $G \subseteq \Qcal$ is an $\Ecal$-generic
filter, then $G$ must intersect every element of $\Ecal$.
Also observe that if $\Scal := \{\{q\} : q \in Q\}$, then the $\Scal$-generic filters are exactly the \emph{ultrafilters} ---
those filters which are maximal.

In order to illustrate the parallel with randomness, take $\Qcal = \Rcal$.
Observe that if $E \subseteq \Rand$ is exhaustive, then its union has full measure.
Conversely, if $E \subseteq \Rand$ is countable and $\bigcup E$ has full measure,
then $E$ is exhaustive.
Thus in this setting, genericity is an assertion that certain measure 1 events occur.

There are two other order-theoretic notions closely related to being exhaustive which it will be useful to define.
A family of pairwise incompatible conditions is said to be an \emph{antichain}.
Notice that this differs from the usual
notion of an antichain in a poset, where antichain would mean pairwise \emph{incomparable}.
Observe that any maximal antichain is exhaustive but that in general exhaustive families need not be pairwise incompatible.
A family $\Dcal$ of conditions is \emph{dense} if every element of $\Qcal$ has an extension in $\Dcal$.
For example, the collection of all elements of $\Rand$ which are compact is dense in $\Rand$.
Observe that, by Zorn's Lemma, every dense set in a partial order contains a maximal antichain.
Two forcings are said to be \emph{equivalent} if they have dense suborders which are isomorphic.
The reason for this is that such forcings generate the same generic extensions.

\section[A theorem of Galvin and Nash-Williams]%
{A precursor to the forcing relation: a partition theorem of Galvin and Nash-Williams}

In this section, we will prove the following theorem of Galvin and Nash-Williams which generalizes
Ramsey's theorem.
The proof is elementary, but crucially employs the forcing relation, albeit implicitly.
We will also use this partition relation in Section \ref{GP:sec}.
The presentation in this section follows \cite[\S 5]{topics_topology:Todorcevic}.
If $A \subseteq \omega$, let $[A]^\omega$ denote all infinite subsets of $A$.

\begin{thm}[see \cite{Galvin-Prikry}]
\label{GNW}
If \(\Fcal\) is a family of nonempty finite subsets of \(\omega\), then there is an infinite subset
\(H\) of \(\omega\) such that either:
\begin{enumerate}[\indent a.]

\item no element of \(\Fcal\) is a subset of \(H\) or 

\item every infinite subset of \(H\) has an initial segment which is in \(\Fcal\).

\end{enumerate}
\end{thm}

Notice that Ramsey's theorem is the special case of this theorem in which all elements of
$\Fcal$ have the same cardinality.
We will now introduce some terminology which will be useful in organizing the proof of Theorem \ref{GNW}.
Fix \(\Fcal\) as in the statement of the theorem.
If \(a \subseteq \omega\), \(A \subseteq \omega\) with \(a\) finite and \(A\) infinite,
then we say that \(A\) \emph{accepts} \(a\) if whenever \(B \subseteq A\) is infinite with \(\max (a) < \min (B)\),
then \(a \cup B\) has an initial segment in \(\Fcal\).
We say that \(A\) \emph{rejects} \(a\) if no infinite subset of \(A\) accepts \(a\) and that
\(A\) \emph{decides} \(a\) if it either accepts or rejects \(A\).
  
We will prove the conclusion of the theorem through a series of lemmas.

\begin{lem}
If \(A\) rejects \(a\), then \(\{k \in A : A \textrm{ accepts } a \cup \{k\}\}\) is finite.
\end{lem}

\begin{proof}
If \(B := \{k \in A : A \textrm{ accepts } a \cup \{k\}\}\) is infinite, then \(B\) is an infinite subset
of \(A\) which accepts \(a\).
\end{proof}

\begin{lem}
There is an infinite set \(H \subseteq \omega\) which decides all of its finite subsets.
\end{lem}

\begin{proof}
Recursively construct infinite sets \(\omega \supseteq H_0 \supseteq H_1 \supseteq \cdots\) such that
if \(n_k:=\min (H_k)\) then \(n_{k-1} < n_k\) and \(H_k\) decides all subsets of \(\{n_i : i < k\}\).
It follows that \(H := \{n_i : i < \infty\}\) decides all finite subsets of \(\omega\).
\end{proof}

\begin{lem}
If \(H \subseteq \omega\) is infinite and decides all of its finite subsets, then either \(H\)
accepts \(\emptyset\) or else there is an infinite \(A \subseteq H\) which rejects all of its finite subsets.
\end{lem}

\begin{proof}
If \(H\) rejects the emptyset and decides all of its finite subsets, then
recursively construct \(n_0 < n_1 < \ldots\) in \(H\) so that for each \(k\), 
\(H\) rejects all subsets of \(\{n_i : i < k\}\).
The choice of the next \(n_k\) is always possible since
\[
\{n : \exists a \subseteq \{n_i : i < k\} (H \textrm{ accepts } a \cup \{n\})\}
\]
is finite.
The set \(A := \{n_i : i < \infty\}\) now rejects all of its finite subsets.
\end{proof}

In order to finish the proof of Theorem \ref{GNW},
observe that if \(H\) accepts the emptyset, then every infinite subset
of \(H\) contains an initial segment in \(\Fcal\).
By the previous lemmas, it therefore suffices to show that if \(A\) is an infinite set which
rejects all of its finite subsets, then no element of \(\Fcal\) is a subset of \(A\).
If \(a \in \Fcal\) with \(a \subseteq A\), then \(B := A \setminus \{0,\ldots,\max (a)\}\) would
accept \(a\), which is impossible.
This finishes the proof of Theorem \ref{GNW}.

\section{The formalism of the forcing relation}

\label{formalism:sec}

In this section we will develop the forcing relation and the forcing language axiomatically, treating the
notion of a \emph{$\Qcal$-name} and the \emph{forcing relation $\forces$} as undefined concepts;
the definitions are postponed until Section \ref{semantics:sec}.
The advantage of this approach is that it emphasizes the aspects of the formalism which
are actually used in practice.

Let $\Qcal$ be a forcing, fixed for the duration of the section.
As we stated earlier, one can view $\Qcal$ as providing the collection of events of positive measure with
respect to some abstract notion of randomness.
In this analogy, a $\Qcal$-name would correspond to a set-valued random variable.
It is conventional to denote $\Qcal$-names by letters with a ``dot'' over them.

There are two examples of $\Qcal$-names which deserve special mention.
The first is the ``check name'':
for each set $x$, there is a $\Qcal$-name $\check x$.
This corresponds to a random variable which is constant --- it does not depend on the outcome.
The other is the $\Qcal$-name $\dot G$ for the generic filter;
this corresponds to the random variable representing the outcome of the random experiment.

The \emph{forcing language} associated to $\Qcal$
is the collection of all first order formulas in the language of set theory augmented
by adding a constant symbol for each $\Qcal$-name.
If $q$ is in $\Qcal$ and $\phi$ is a sentence in the forcing language, then informally
the forcing relation $q \forces \phi$ asserts that
if the event corresponding to $q$ occurs, then \emph{almost surely} $\phi$ will be true.
In the absence of the definitions of ``$\Qcal$-name'' and ``$\forces$,'' the following
properties can be regarded as axioms which govern the behavior of these primitive concepts.
They can be proved from the definitions of $\Qcal$-names and the forcing relation which will be given
in Section \ref{semantics:sec}.

\begin{property} \label{check_names}
For any $p \in \Qcal$ and any sets $x$ and $y$:
\begin{enumerate}[\indent a.]

\item
$p \forces \check x \in \check y$ if and only if $x \in y$;

\item
$p \forces \check x = \check y$ if and only if $x = y$;

\end{enumerate}
\end{property}

\begin{property} \label{filter_name}
For $p, q \in \Qcal$, $p \forces \check q \in \dot G$ if and only if whenever $r \in \Qcal$ is compatible with $p$,
$r$ is compatible with $q$.
\end{property}

If $\Qcal$ is separative, then this property takes a simpler form: $p \forces \check q \in \dot G$ if and only
if $p \leq q$.

\begin{property} \label{ordinal_names}
If $\dot \alpha$ is a $\Qcal$-name, $p \in \Qcal$, and
$p \forces \dot \alpha \textrm{ is an ordinal}$, then there is an ordinal $\beta$ such that
$p \forces \dot \alpha \in \check \beta$.
\end{property}

It is useful to define the following terminology:
if there is a $z$ such that $q \forces \dot y = \check z$,
then we say that $q$ \emph{decides} $\dot y$ (to be $z$).
Similarly, if $p \forces \phi$ or $p \forces \neg \phi$, then we say that $p$ \emph{decides} $\phi$.

\begin{property} \label{decide}
For any $x$, any $\Qcal$-name $\dot y$, and $p \in \Qcal$, 
if $p \forces \dot y \in \check x$, then
there is a $q \leq p$ which decides $\dot y$.
\end{property}

\begin{property}
\label{collection_for_names}
If $\dot x$ is a $\Qcal$-name and $p \in \Qcal$, then the collection of all $\Qcal$-names $\dot y$ such that
$p \forces \dot y \in \dot x$ forms a set and the collection of all $\Qcal$-names $\dot y$ such that
$p \forces \dot y = \dot x$ forms a set.

\end{property}

\begin{remark}
Unlike the other properties, this one is dependent on
the definition of $\Qcal$-name which we will give in the next section.
\end{remark}

\begin{property} \label{negation_monotone}
If $p \in \Qcal$ and $\phi$ is a formula in the forcing language, then
$p \forces \neg \phi$ if and only if there is no $q \leq p$ such that $q \forces \phi$.
\end{property}

Observe that this property implies that
if $p \forces \phi$ and $q \leq p$, then $q \forces \phi$.
Property \ref{negation_monotone} can be seen as providing an organizational tool in the proof of Theorem \ref{GNW}:
if $\Qcal := ([\omega]^\omega,\subseteq)$ then
an $A \in [\omega]^\omega$ accepts $a$ if and only if $A$ forces that every element of the generic filter contains
an infinite set with an initial part in $\check \Fcal$.
An infinite $A$ rejects $a$ if it forces the negation of this assertion.

\begin{property} \label{completeness_of_names}
If $p \in \Qcal$, then $p \forces \exists v \phi$ if and only if there is a $\Qcal$-name $\dot x$ such that
$p \forces  \phi [\dot x/v]$. 
\end{property}

\begin{property} \label{ZFC_forced}
For any $q \in \Qcal$, the collection of sentences in the forcing
language which are forced by $q$ contains the ZFC axioms, the axioms of first order logic,
and is closed under \emph{modus ponens}.
Moreover, if the axioms of ZFC are consistent, then so are the sentences forced by $q$. 
\end{property}

If $\one \forces_\Qcal \phi$, then we will sometimes say that ``$\Qcal$ forces $\phi$'' or, if $\Qcal$ is clear from the context,
that ``$\phi$ is forced.''
Similarly, we will write ``$\dot x$ is a $\Qcal$-name for...'' to mean
``$\dot x$ is a $\Qcal$-name and $\Qcal$ forces that $\dot x$ is...''.

In order to demonstrate how these properties can be used, we will prove
the following useful propositions.

\begin{prop} \label{check_quant}
Suppose that $x$ is a set and $\phi(v)$ is a formula in the forcing language.
If for all $y \in x$,
$p \forces \phi(\check y)$,
then $p \forces \forall y \in \check x \phi (y)$.
\end{prop}

\begin{proof}
We will prove the contrapositive.
Toward this end, suppose that $p$ does not force $\forall y \in \check x \phi(y)$.
It follows from Property \ref{negation_monotone}
there is a $q \leq p$ such that $q \forces \neg \forall y \in \check x \phi(y)$.
By Property \ref{ZFC_forced}, this is equivalent to $q \forces \exists y \in \check x \neg \phi(y)$.
By Property \ref{completeness_of_names},
there is a $\Qcal$-name $\dot y$ such that
\[
q \forces (\dot y \in \check x) \land (\neg \phi( \dot y)).
\]
By Property \ref{ZFC_forced} $q \forces \dot y \in \check x$ and
therefore by Property \ref{decide}, there is a $r \leq q$ and a $z$ in $x$ such that
$r \forces \dot y = \check z$.
But now, by Property \ref{ZFC_forced}, $r \forces \neg \phi(\check z)$ and hence
by Property \ref{negation_monotone}
$p$ does not force $\phi(\check z)$.
\end{proof}

\begin{prop} \label{bounded_abs}
Suppose that $\phi(v_1,\ldots,v_n)$ is a formula in the language of set theory
with only bounded quantification.
If $x_1,\ldots,x_n$ are sets 
and $\phi(x_1,\ldots,x_n)$ is true, then $\one \forces \phi(\check x_1,\ldots,\check x_n)$.
\end{prop}

\begin{proof}
The proof is by induction on the length of $\phi$.
If $\phi$ is atomic, then this follows from Property \ref{check_names}.
If $\phi$ is a conjunct, disjunct, or a negation,
then the proposition follows from Property \ref{ZFC_forced}
and the induction hypothesis.
Finally, suppose $\phi(v_1,\ldots,v_n)$ is of the form $\forall w \in v_n \psi(v_1,\ldots,v_n,w)$.
If $\forall w \psi(x_1,\ldots,x_n,w)$ is true, then for each $w$,
$\psi(x_1,\ldots,x_n,w)$ is true.
By our induction hypothesis, $\one \forces \phi(\check x_1,\ldots,\check x_n,\check w)$ for each
$w$.
By Proposition \ref{check_quant}, it follows that
$\one \forces \forall w \in \check x_n \psi(\check x_1,\ldots,\check x_n,w)$.
\end{proof}

\begin{prop} \label{WF_abs}
Suppose that $T$ is a set consisting of finite length sequences, closed under taking initial segments.
If there is a forcing $\Qcal$ and some 
$q \in \Qcal$ forces ``there is an infinite sequence $\sigma$, all of whose
finite initial parts are in $T$,'' then such a sequence $\sigma$ exists.
\end{prop}

\begin{proof}
If no such sequence $\sigma$ exists, then there is a function $\rho$
from $T$ into the ordinals such that if
$s$ is a proper initial segment of $t$, then $\rho(t) \in \rho(s)$.
Such a $\rho$ certifies the nonexistence of such a $\sigma$ since such a $\sigma$ would define
a strictly decreasing infinite sequence of ordinals.
Observe that the assertion that $\rho$ is a strictly decreasing map from $T$ into the ordinals
is a statement about $\rho$ and $T$ involving only bounded quantification.
By Proposition \ref{bounded_abs}, this statement is forced by every forcing $\Qcal$.
\end{proof}

There is a special class of forcings for which there is a more  
conceptual picture of the forcing relation.
We begin by stating a general fact about forcings.
\begin{thm} \label{cBA}
For every forcing $\Qcal$, $\Qcal$ is isomorphic to a dense suborder of the positive
elements of a complete Boolean algebra.
\end{thm}
\noindent
Here we recall that a Boolean algebra is \emph{complete} if every subset has a least upper
bound.
A typical example of a complete Boolean algebra is the algebra of measurable subsets
of $[0,1]$ modulo the ideal of measure zero sets.
The algebra of Borel subsets of $[0,1]$ modulo the ideal of first category sets is similarly a complete Boolean algebra.
Random and Cohen forcing, respectively, are isomorphic to dense suborders of the positive elements of these
complete Boolean algebras.

Suppose now that $\Qcal$ is the positive elements of some complete Boolean algebra $\Bcal$.
If $\phi$ is a formula in the forcing language, then define the \emph{truth value}
$\truth{\phi}$
of $\phi$ to be the least upper bound of all $b \in \Bcal$ such that $b \forces \phi$.
Observe that if $a \leq \truth{\phi}$, then $a$ cannot force
$\neg \phi$.
Hence $\truth{\phi}$ forces $\phi$.
The rules which govern the logical connectives now take a particularly nice form:
\[
\truth{\neg \phi} = \truth{\phi}^{\C}
\qquad
\truth{\phi \land \psi} = \truth{\phi} \land \truth{\psi}
\qquad
\truth{\phi \lor \psi} = \truth{\phi} \lor \truth{\psi}
\]
\[
\truth{\forall v \phi} = \bigmeet_{\dot x} \truth{\phi[\dot x/v]}
\qquad
\truth{\exists v \phi} = \bigjoin_{\dot x} \truth{\phi[\dot x/v]}
\]
Notice that while $\dot x$ ranges over all $\Qcal$-names in the last equations
--- a proper class --- the collection of all possible values of
$\truth{\phi[\dot x/v]}$ is a set and therefore the last items are
meaningful.

In spite of the usefulness of complete Boolean algebras
in understanding forcing and also in some of
the development of the abstract theory of forcing,
forcings of interest rarely present themselves as complete Boolean algebras (the notable exceptions
being Cohen and random forcing).
While Theorem \ref{cBA} allows us to represent any forcing inside a complete
Boolean algebra, defining forcing strictly in terms of complete
Boolean algebras would prove cumbersome in practice.

\section{Names, interpretation, and semantics}

\label{semantics:sec}

In this section we will turn to the task of giving a
formal definition of what is meant by a \emph{$\Qcal$-name} and
$q \forces \phi$.
This will in turn be used to give a semantic perspective of forcing.
The definitions in this section are not essential for understanding
most forcing arguments and the reader may wish to skip this section on their first reading
of the material.
Others, however, may wish to have a tangible model of the axioms.

Before proceeding, we need to recall the notion of the \emph{rank} of a set.
If $x$ is a set, then the \emph{rank} of $x$ is defined recursively:
the rank of the emptyset is $0$ and the rank of $x$ is the least ordinal which is strictly greater than the ranks of
its elements.
This is always a well defined quantity and it will sometimes be necessary to give definitions by recursion on rank.
We recall that formally an ordered pair $(x,y)$ is defined to be $\{x,\{x,y\}\}$;
this is only relevant in that the rank of $(x,y)$ is greater than the ranks of
either $x$ or $y$.

Now let $\Qcal$ be a forcing, fixed for the duration of the section.
If $q \in \Qcal$, let $\Qcal_q$ denote the forcing
$(\{r \in \Qcal : r  \leq q\},\leq)$.

\begin{defn}[name]
A set $\dot x$ is a \emph{$\Qcal$-name}
if every element of $\dot x$ is of the form $(\dot y,q)$ where
$\dot y$ is a $\Qcal_q$-name and $q$ is in $\Qcal$.
\end{defn}
\noindent
(The requirement that $\dot y$ be a $\Qcal_q$-name is to help ensure that
Property \ref{collection_for_names} is satisfied.)
Notice that this apparently implicit definition is actually a definition by recursion
on \emph{rank}, as discussed in Section \ref{prelim:sec}.
Furthermore, if $\Pcal \subseteq \Qcal$ is a suborder, then any $\Pcal$-name
is also a $\Qcal$-name.

The following provide two important examples of $\Qcal$-names.

\begin{defn}[check names]
If $x$ is a set, $\check x$ is defined recursively by
\[
\{(\check y,\one) : y \in x\}.
\]
\end{defn}

\begin{defn}[name for the generic filter]
$\dot G := \{(\check q,q) : q \in \Qcal\}$.
\end{defn}

As mentioned in the previous section, the notion of a $\Qcal$-name is intended to describe a procedure
for building a new set from a given filter $G \subseteq \Qcal$.
This procedure is formally described as follows.
 
\begin{defn}[interpretation]
If $G$ is any filter and $\dot x$ is any $\Qcal$-name, define $\Int{\dot x}{G}$ recursively by
\[
\Int{\dot x}{G} := \{\Int{\dot y}{G} : \exists p \in G\ ((\dot y,p) \in \dot x)\}
\]
\end{defn}
\noindent
Again, this is a definition by recursion on rank.
In the analogy with randomness, $\Int{\dot x}{G}$ corresponds to evaluating a random variable
at a given outcome.

The following gives the motivation for the definitions of $\check x$ and $\dot G$.

\begin{prop}
If $H$ is any filter and $x$ is any set,
then $\Int{\check x}{H} = x$.
\end{prop}

\begin{prop}
If $H$ is any filter, then $\Int{\dot G}{H} = H$.
\end{prop}

\begin{remark}
It is possible to define \emph{$\Qcal$-name} to just be a synonym for \emph{set}.
The definition of $\Int{\dot x}{G}$ would be left unchanged so that only those elements
of $\dot x$ which are ordered pairs with a second coordinate in $\Qcal$ play any role in
the interpretation.
This alternative has the advantage of brevity and 
much of what is stated in the previous section remains true with this alteration.
On the other hand, it is easily seen that Property \ref{collection_for_names} fails.
For instance those sets which do not contain any ordered pairs forms a proper class and each member
of this class is forced by the trivial condition to be equal to the emptyset.
\end{remark}

We now turn to the formal definition of the forcing relation.
The main complexity of the definition of the forcing relation is tied
up in the formal definition of $p \forces \dot x \in \dot y$.

\begin{defn}[forcing relation: atomic formulae]
If $\Qcal$ is a forcing and $\dot x$ and $\dot y$ are $\Qcal$-names,
then we define the meaning of $p \forces \dot x = \dot y$ and $p \forces \dot x \in \dot y$ as follows
(the definition is by simultaneous recursion on rank):
\begin{enumerate}[\indent a.]

\item $p \forces \dot x = \dot y$ if and only if
for all $\dot z$ and $p' \leq p$,
\[
(p' \forces \dot z \in \dot x) \leftrightarrow (p' \forces \dot z \in \dot y).
\]

\item $p \forces \dot x \in \dot y$ if and only if for every $p' \leq p$ there is
a $p'' \leq p'$ and a $(\dot z,q)$ in $\dot y$ such that
$p'' \leq q$ and $p'' \forces \dot x = \dot z$.

\end{enumerate}
\end{defn}

Notice that the definition of $p \forces \dot x = \dot y$ is precisely to ensure
that the \emph{Axiom of Extensionality} --- which asserts that two sets are equal if they have the same
set of elements --- is forced by any condition.
The definition of the forcing relation for nonatomic formulas
is straightforward and is essentially determined by the properties of the forcing relation
mentioned already in Section \ref{formalism:sec}.

\begin{defn}[forcing relation: logical connectives]
Suppose that $p \in \Qcal$ and $\phi$ and $\psi$ are formulas in the forcing language.
The following are true:
\begin{enumerate}[\indent a.]

\item $p \forces \neg \phi$ if there does not exist a $q \leq p$ such that
$q \forces \phi$.

\item $p \forces \phi \land \psi$ if and only if $p \forces \phi$ and $p \forces \psi$.

\item $p \forces \phi \lor \psi$ if there does not exist a $q \leq p$ such that
$q \forces \neg \phi \land \neg \psi$.

\item $p \forces \forall v \phi$ if and only if for all $\dot x$,
$p \forces \phi[\dot x/v]$.

\item $p \forces \exists v \phi$ if and only if there is an $\dot x$ such that
$p \forces \phi[\dot x/v]$.

\end{enumerate}
\end{defn}
The interested reader may wish to stop and verify that the definitions of
$\forces_\Qcal$ and $\Qcal$-name given in this section satisfy the properties stated in 
Section \ref{formalism:sec}.

The following theorem is one of the fundamental results about forcing.
It connects the syntactic properties of the forcing relation with truth
in generic extensions of models of set theory.
If $M$ is a countable transitive model of ZFC, $\Qcal$ is a forcing in $M$, and
$G \subseteq \Qcal$ is an $M$-generic filter,
define 
\[
M[G] := \{\Int{\dot x}{G} : \dot x \in M \mand \dot x \textrm{ is a $\Qcal$-name}\}.
\] 
In this context, $M[G]$ is the \emph{generic extension} of $M$ by $G$ and $M$ is referred
to as the \emph{ground model}.
Notice that
\[
M = \{\Int{\check x}{G} : x \in M\} \subseteq M[G] \qquad \textrm{and} 
\qquad G = \Int{\dot G}{G} \in M[G].
\]
The following
theorem relates the semantics of forcing (i.e. truth in the generic extension) with
the syntax (i.e. the forcing relation).

\begin{thm}
Suppose that $M$ is a countable transitive model of ZFC and that $\Qcal$
is a forcing which is in $M$.
If $q$ is in $\Qcal$, $\phi(v_1,\ldots,v_n)$ is a formula in the language of set theory,
and $\dot x_1,\ldots,\dot x_n$
are in $M$, then the following are equivalent:
\begin{enumerate}[\indent a.]

\item $q \forces \phi(\dot x_1,\ldots,\dot x_n)$.

\item 
$M[G] \models \phi(\Int{\dot x_1}{G},\ldots,\Int{\dot x_n}{G})$
whenever $G \subseteq \Qcal$ is an $M$-generic filter and $q$ is in $G$.

\end{enumerate}
\end{thm}
\begin{remark}
This theorem can be modified to cover countable transitive models of
sufficiently large finite fragments of ZFC.
In fact this is crucial if one wishes to give a rigorous treatment of the semantics.  
By G\"odel's second incompleteness theorem, ZFC alone does
not prove that there are any set models of ZFC (countable or otherwise).
This is in fact our main reason for de-emphasizing the semantics: while it is formally necessary
to work with models of finite fragments of ZFC, this only introduces technicalities which are
inessential to understanding what can be achieved with forcing.
\end{remark}

While we will generally not work with the semantics of forcing, let us note here
that it is conventional to use $\dot x$ to denote a $\Qcal$-name for
an element $x$ of a generic extension $M[G]$.
While such names are not unique, the choice generally does not matter and
this informal convention affords a great deal of notational economy.

We will now finish this section with some further discussion and notational
conventions concerning names.
It is frequently the case in a forcing construction that one encounters a $\Qcal$-name
for a function $\dot f$ whose domain is forced by some condition to be a ground model set;
that is, for some set $D$,  $p \forces \dom(\dot f) = \check D$.
A particularly common occurrence is when $D = \omega$ or, more generally, some ordinal.
Under these circumstances, it is common to abuse notation and regard $\dot f$ as a function
defined on $D$, whose values are themselves names: $\dot f(x)$ is a $\Qcal$-name $\dot y$ such that
it is forced that $\dot f(\check x) = \dot y$.
Notice that if, for some sets $A$ and $B$, $p \forces \dot f: \check A \to \check B$,
it need not be the case that $\dot f(a)$ is of the form $\check b$ for some $b$ in $B$ --- i.e.
$p$ need not \emph{decide} the value of $\dot f(a)$ for a given $a \in A$.

In most cases, names are not constructed explicitly.
Rather a procedure is described for how to build the object to which the name
is referring.
Properties \ref{completeness_of_names} and \ref{ZFC_forced} are then implicitly invoked.
For example, if $\dot x$ is a $\Qcal$-name, $\bigcup \dot x$ is the $\Qcal$-name for the unique set
which is forced to be equal to the union of $\dot x$.
Notice that there is an abuse of notation at work here: formally, $\dot x$ is a set which has a union $y$.
It need not be the case that $y$ is even a $\Qcal$-name and certainly one should not
expect $\one \forces \bigcup \dot x = \check y$.
This is one of the reasons for using ``dot notation'': it emphasizes the role of the object as a name.

A more typical example of is $\omega_1$, the least uncountable ordinal.
Since ZFC proves ``there is a unique set $\omega_1$ such that $\omega_1$ is an ordinal, $\omega_1$ is uncountable,
and every element of $\omega_1$ is countable,''
it follows that if $\Qcal$ is any forcing, $\one \forces_\Qcal \exists x \phi(x)$, where $\phi(x)$ asserts $x$ is
the least uncountable ordinal.
In particular there is a $\Qcal$-name $\dot x$ such that $\one \forces_\Qcal \phi(\dot x)$.
Unless readability dictates otherwise, such names are denoted by adding a ``dot'' above the usual notation
(e.g. $\dot \omega_1$).

Another example is $\R$.
Recall that $\R$ is the completion of $\Q$ with respect to its metric --- formally the collection
of all equivalence classes of Cauchy sequences of rationals.
We use this same formal definition of $\R$ to define $\dot \R$:
if $\Qcal$ is a forcing, $\dot \R$ is the collection of all $\Qcal$-names for equivalence classes of
Cauchy sequences of rational numbers.
Notice that $\dot \R$ is not the same as $\check \R$
and, more to the point, we need not even have that
$\one \forces_{\Qcal} \dot \R = \check \R$ for a given forcing $\Qcal$.
This construction also readily generalizes to define $\dot X$ if $X$ is a complete
metric space.
The $\Qcal$-name $\dot X$ is then the collection of all $\Qcal$-names $\dot x$ such that
$\one$ forces that $\dot x$ is an equivalence class of Cauchy sequences
of elements of $\check X$.
That is, $\dot X$ is a $\Qcal$-name for the completion of $\check X$.

Finally, there are some definable sets which are always interpreted as
ground model sets and do not depend on the generic filter.
Two typical examples are finite and countable ordinals such as $0$, $1$, $\omega$, and $\omega^2$ as well
as sets such as $\Q$.
In such cases, checks are suppressed in writing the names for ease of readability --- we will write
$\Q$ and not $\check \Q$ or $\dot \Q$ in formulae which occur in the forcing language.

\section{The cast}

\label{cast:sec}

We will now introduce the examples which we will put to work
throughout the rest of the article.
The first class of examples provides the justification
for viewing forcings as abstract notions of randomness.

\begin{example}[random forcing]
Define $\Rand$ to be the collection of all measurable subsets of $[0,1]$ which have positive measure.
If $I$ is any index set, let $\Rand_I$ denote the collection of all measurable subsets of
$[0,1]^I$ which have positive measure.
Here $[0,1]$ is equipped with Lebesgue measure and $[0,1]^I$ is given the product measure.
Define $q \leq p$ to mean $q \subseteq p$.
This order is not separative so formally here we define $\Rand$ and $\Rand_I$ to be the corresponding
separative quotients.
This amounts to identifying those measurable sets which differ by a measure zero set.
Notice that every element of $\Rand_I$ contains a compact set in $\Rand_I$ --- the compact elements of $\Rcal_I$
are dense.
Furthermore, any two elements
of $\Rand_I$ are compatible if their intersection has positive measure.
\end{example} 

When working with a forcing $\Qcal$, one is rarely interested in the generic filter itself but rather
in some \emph{generic object} which can be derived in some natural way from the generic filter.
For instance, in $\Rand_I$ it is forced that
\[
\bigcap \{ \cl(q) : q \in \dot G\}
\]
contains a unique element.
We will let $\dot r$ denote a fixed $\Rand_I$-name for this element.
For each $i \in I$, let $\dot r_i$ denote a fixed $\Rand_I$-name for the $i$th coordinate of $\dot r$ and
observe that for all $i \ne j$ in $I$,  
\[
D_{i,j} := \{q \in \Rand_I : (x \in \cl(q)) \rightarrow (x(i) \ne x(j))\}
\]
is dense.
Therefore $\one \forces_{\Rand_I} \forall i \ne j \in \check I\ (\dot r_i \ne \dot r_j)$.
In particular, it is forced by $\Rand_I$ that $|\dot \R| \geq |\check I|$.
(Notice however, that we have not established that if, e.g., $I = \aleph_2$,
then $\one \forces_{\Rand_I} \dot \aleph_2 = \check \aleph_2$.
This will be established in Section \ref{ccc:sec}.)

In the context of $\Rand$, we will use $\dot r$ to denote a $\Rand$-name for the unique element
of $\bigcap \{\cl (q) : q \in \dot G\}$.
If $M$ is a transitive model of ZFC, then $r \in [0,1]$ is in every measure 1 Borel set coded in $M$ if and only if
$\{q \in \Rand \cap M : r \in q\}$ is a $M$-generic filter.
Such an $r$ is commonly referred to as a \emph{random real} over $M$.
The notion of a random real was first introduced by Solovay \cite{solovay_model}.

The next class of examples includes Cohen's original forcing from \cite{CON_negCH}.
Just as random forcing is rooted in measure theory,
Cohen forcing is rooted in the notion of Baire category.

\begin{example}[Cohen forcing]
Let $\Cohen$ denote the collection of all \emph{finite partial functions} from $\omega$ to $2$:
all functions $q$ such that the domain of $q$ is a finite subset of $\omega$ and the range of $q$ is
contained in $2 = \{0,1\}$.
We order $\Cohen$ by $q \leq p$ if $q$ extends $p$ as a function.
If $I$ is a set, let $\Cohen_I$ denote the collection
of all finite partial functions from $I \times \omega$ to $2$, similarly ordered by extension.
It is not difficult to show that $\Cohen_I$ is isomorphic to a dense suborder of the collection of all
nonempty open subsets of $[0,1]^I$, ordered by containment. 
This makes $\Cohen_I$ analogous to $\Rand_I$ (in fact it is a suborder), although viewing
$\Cohen_I$ as a collection of finite partial functions will often be more convenient from the point
of view of notation.
\end{example}

It is very often the case that forcings consist of a collection of \emph{partial functions} ordered
by \emph{extension}.
By this we mean that $q \leq p$ means that $p$ is the restriction of $q$ to the domain of $p$.
A filter in the forcing is then a collection of functions which is directed under containment
and whose union is therefore also a function.
This union is the generic object derived from the generic filter.

In the case of $\Cohen_I$, observe that
for each $i \ne j$ in $I$ and $n < \omega$, both 
\[
\{q \in \Cohen_I : (i,n) \in \dom(q)\}
\]
and
\[
\{q \in \Cohen_I : \exists m\ \big((\{(i,m),(j,m)\} \subseteq \dom(q)) \land (q(i,m) \ne q(j,m)) \big)\}
\]
are dense.
In particular, the generic object will be a function from $I \times \omega$ into $2$.
As in the case of $\Rand_I$,
such a generic object naturally corresponds to an indexed family $\Seq{r_i : i \in I}$
of elements of $[0,1]$ and
genericity ensures that these elements are all distinct.

If $M$ is a transitive model of ZFC, then $r \in [0,1]$ is in every dense open set coded in $M$ if and only if the set of finite restrictions of the binary expansion of $r$ is an
$M$-generic filter for the forcing $\Cohen$.
Such an $r$ is commonly referred to as a \emph{Cohen real} over $M$.
Notice that $[0,1]$ is a union of a measure 0 set and a set of first category:
for every $n$, the rationals in $[0,1]$ are contained in a relatively dense open set
of measure less than $1/n$.
Thus no element of $[0,1]$ is both
a Cohen real and a random real over a transitive model of ZFC.
In fact there are qualitative difference between Cohen and random reals as well.
For instance, in the case of random forcing it is forced that
\[
\lim_{n \to \infty} \frac{1}{n} |\{i < n : r(i) = 1\}| = \frac{1}{2}
\]
where as in the case of Cohen forcing, it is forced that the limit does not exist.

The next example of a forcing appears similar at first to random forcing, but in fact it is quite
different in nature.

\begin{example}[Amoeba forcing]
If $1 > \epsilon > 0$, then define $\Ameoba_\epsilon$ to be the collection of all elements of
$\Rand$ of measure greater than $\epsilon$.
This is regarded as a forcing with the order induced from $\Rand$.
\end{example}

Notice that the compatibility relation on $\Ameoba_{\epsilon}$ differs from that inherited from $\Rand$:
two conditions in $\Ameoba_{\epsilon}$ are compatible in $\Ameoba_{\epsilon}$ if and only if their intersection has measure
greater than $\epsilon$.

The previous forcings all introduce a new subset of $\omega$ to the ground model.
The next example adds a new ultrafilter on $\omega$ but, as we will see in Section \ref{sigma-closed:sec}, it
does not introduce a new subset of $\omega$.

\begin{example} 
Let $[\omega]^\omega$ denote the collection
of all infinite subsets of $\omega$.
The ordering of containment on $[\omega]^\omega$ is not separative;
its separative quotient is obtained by identifying those $x$ and $y$
which have finite symmetric difference.
We will abuse notation and denote this quotient by $[\omega]^{\omega}$ as well.
\end{example}

The next forcing was introduced by Mathias \cite{happy_families} to study
infinite dimensional generalizations of Ramsey's theorem.

\begin{example}[Mathias forcing]
Let $\Mathias$ denote the collection of all pairs $p = (a_p,A_p)$
such that $A_p$ is in $[\omega]^{\omega}$ and $a_p$ is a finite initial part of $A_p$.
Define $q \leq p$ to mean $a_p \subseteq a_q$ and $A_q \subseteq A_p$.
Note in particular that in this situation $a_p$ is an initial segment of $a_q$ and
$a_q \setminus a_p$ is contained in $A_p$.
This forcing is known as \emph{Mathias forcing}.
\end{example}

The final example is an illustration of the potential raw power of forcing.
Typically the phenomenon of collapsing cardinals to $\aleph_0$ is something one wishes
to avoid.

\begin{example}[collapsing to $\aleph_0$]
If $X$ is a set, consider the collection $X^{<\omega}$ of all finite sequences of elements
of $X$, ordered by extension.
Observe that if $x$ is in $X$, then the collection of all elements of $X^{<\omega}$ which
contain $x$ in their range is dense.
Thus $X^{<\omega}$ forces that $|\check X| = \aleph_0$.
Notice that if $X = \R$ in this example, then it is forced that
$|\check \R| = \check \aleph_0 <  |\dot \R|$.
\end{example}

\section{The countable chain condition}

\label{ccc:sec}

Something which is often an important consideration in the analysis of a forcing
is whether uncountability is preserved.
That is, does $\one \forces_\Qcal \check \aleph_1= \dot \aleph_1$?
More generally, one can ask whether cardinals are preserved by forcing with $\Qcal$:
if $X$ and $Y$ are sets such that $|X| < |Y|$,
then does $\one \forces_\Qcal |\check X| < |\check Y|$?

In general, this can be a very subtle matter (and even can be influenced by forcing).
One way to demonstrate that a forcing preserves cardinals
is to verify that it satisfies the \emph{countable chain condition (c.c.c.)}.
A forcing $\Qcal$ satisfies the \emph{c.c.c.} if every family of pairwise incompatible elements of $\Qcal$ is at most countable.
The following proposition dates to Cohen's proof that the Continuum Hypothesis is independent of ZFC.

\begin{prop} \label{ccc_card}
Suppose that $\Qcal$ is a c.c.c. forcing.
If $\kappa$ is a regular cardinal, then $\one \forces \check \kappa \textrm{ is a regular cardinal}$.
In particular, $\kappa$ is a cardinal then $\one \forces \kappa \textrm{ is a cardinal}$ and hence
for every ordinal $\alpha$,
$\one \forces \dot \aleph_{\alpha} = \check \aleph_{\alpha}$. 
\end{prop}

\begin{proof}
The second conclusion follows from the first since every
cardinal is the supremum of a set of successor cardinals and every
supremum of a set of successor cardinals is a cardinal.
Let $\kappa$ be a regular cardinal and $\Qcal$ be a given forcing.
Suppose that $\dot f$ and $\dot \lambda$ are $\Qcal$-names and that
$p$ is an element of $\Qcal$
such that
\[
p \forces (\dot \lambda \in \check \kappa) \land (\dot f : \dot \lambda \to \check \kappa)
\]
By extending $p$ if necessary, we may assume without loss of generality that $\dot \lambda = \check \lambda$
for some $\lambda < \kappa$.
It is sufficient to show that $p$ forces that $\dot f$ is not a surjection.
If $\kappa$ is countable, then $\lambda$ is finite and it is possible to decide $\dot f$ by deciding its values
one at a time
(this does not require that $\Qcal$ is c.c.c.).
Thus we will assume that $\kappa$ is uncountable.

For each $\alpha < \lambda$, define
\[
F(\alpha) := \{ \beta < \kappa : \exists q \leq p (q \forces \dot f(\alpha) = \check \beta)\}.
\]
Notice that if $\beta \ne \beta'$ are in $F(\alpha)$ and
$q$ forces that $\dot f(\alpha) = \check \beta$ and $q'$ forces that $\dot f( \alpha) = \check \beta'$,
then $q$ and $q'$ are incompatible (otherwise any extension $\bar q$ would
force $\check \beta = \dot f( \alpha) = \check \beta'$).
Since $\Qcal$ is c.c.c., $F(\alpha)$ is countable and has an upper bound $g(\alpha) < \kappa$.

Since $\kappa$ is regular, the range of $g$ is bounded.
We are therefore finished once we show that
\[
p \forces \forall \alpha \in \check \lambda\ (\dot f( \alpha) \in \check F( \alpha)).
\]
By Proposition \ref{check_quant}, this is equivalent to showing that for all $\alpha$ in $\lambda$,
$p \forces \dot f( \alpha) \in \check F( \alpha)$.
Suppose for contradiction that this is not the case.
Then there is a $q \leq p$ and an $\alpha$ in $\lambda$ such that
$q \forces \dot f( \alpha) \not \in \check F( \alpha)$.
By Proposition \ref{decide}, there is a $q' \leq q$ and a $\beta < \kappa$ such that
$q' \forces \dot f( \alpha) = \check \beta$.
But now $\beta \in F(\alpha)$, a contradiction.
\end{proof}

\begin{prop} \label{generic_prop_K}
Suppose that $\Qcal$ is a c.c.c. forcing and that $\Seq{q_\xi : \xi < \omega_1}$ is a sequence of conditions in $\Qcal$.
Then there is a $p$ such that
\[
p \forces \{\xi \in \check \omega_1 : q_\xi \in \dot G\} \textrm{ is uncountable}.
\]
\end{prop}

\begin{remark}
Notice that this characterizes the c.c.c.: if $A$ is an uncountable antichain in a forcing $\Qcal$,
then any condition forces that $\check A \cap \dot G$ contains at most one element.
\end{remark}

\begin{proof}
Suppose that this is not the case.
Then 
\[
\one \forces
\exists \beta \in \check \omega_1
\forall \xi \in \check \omega_1\ 
(q_\xi \in \dot G \rightarrow \xi < \beta).
\]
By Property \ref{completeness_of_names} of the forcing relation,
there is a $\Qcal$-name $\dot \beta$ for an element of $\omega_1$ such that
\[
\one \forces \forall \xi \in \check \omega_1\ (q_\xi \in \dot G \rightarrow \xi < \dot \beta).
\]
As in the proof of Proposition \ref{ccc_card},
the set of $\alpha < \omega_1$ such that, for some $q \in Q$,
$q \forces \dot \beta = \check \alpha$
is countable and therefore bounded by some $\gamma$.
That is 
$$
\one \forces \forall \xi \in \check \omega_1\ (q_\xi \in \dot G \rightarrow \xi < \check \gamma).
$$
But now $q_\gamma$ forces that $\check q_\gamma$ is in $\dot G$, a contradiction.
\end{proof}

We will now return to some of the examples introduced in Section \ref{cast:sec}.

\begin{prop} \label{cohen_random_ccc}
For any index set $I$, both $\Rcal_I$ and $\Ccal_I$ are c.c.c. forcings.
\end{prop}

\begin{proof}
In the case of $\Rand_I$, this is just a reformulation of the assertion
that if $\Fcal$ is an uncountable
family of measurable subsets of $[0,1]^I$, each having positive measure,
then there are two elements of $\Fcal$ which intersect in a set of positive measure.
The reason for this is that if $\Fcal$ is uncountable, then for some $\epsilon > 0$ there 
are more than $1/\epsilon$ elements of $\Fcal$ with measure at least $\epsilon$.
At least two of these elements must intersect in a set of positive measure.
The same argument applies to $\Ccal_I$, by observing that
we may view $\Ccal_I$ as a dense suborder of the collection of all nonempty open subsets
of $[0,1]^I$, ordered by containment.
Since any nonempty open subset of $[0,1]^I$ has positive measure we may view
$\Ccal_I$ as a suborder of $\Rcal_I$.
Moreover, conditions $p,q \in \Ccal_I$ which are compatible in $\Rcal$ are compatible in $\Ccal_I$.
\end{proof}

\begin{prop} \label{ameoba_ccc}
The forcing $\Ameoba_\epsilon$ satisfies the c.c.c. for every $\epsilon > 0$.
\end{prop}

\begin{proof}
Let $D$ denote the collection of all elements of $\Ameoba_\epsilon$ which are finite unions of rational
intervals.
Notice that $D$ is countable.
For each $p$ in $D$, let $\Fcal_p$ denote the collection of all elements $q$ of $\Ameoba_\epsilon$ such that
$q \subseteq p$ and
\[
\lambda (p) - \lambda (q) < \frac{\lambda(p) - \epsilon}{2}.
\]
Notice that $\bigcup_{p \in D} \Fcal_p$ contains all of the compact sets in $\Ameoba_\epsilon$ which
are in turn dense in $\Ameoba_\epsilon$.
Moreover, any two elements of $\Fcal_p$ intersect in a set
of measure greater than $\epsilon$ and hence have a common lower bound in $\Ameoba_\epsilon$.
If  $X \subseteq \Ameoba_\epsilon$ is uncountable, two distinct elements of $X$ must have extensions in the same
$\Fcal_p$ for some $p$ and thus be compatible.
Hence any antichain in $\Ameoba_\epsilon$ is countable and $\Ameoba_\epsilon$ is c.c.c..
\end{proof}

\begin{remark}
The reader may wonder why we have not bothered to generalize $\Ameoba_\epsilon$ to a larger index set,
given that we did this for $\Rand$.
The reason is that, for uncountable index sets, the analog of $\Ameoba_{\epsilon}$ is not c.c.c.
and in fact collapses the cardinality of $I$ to become countable.
\end{remark}

We finish this section by demonstrating that the Continuum Hypothesis isn't provable
within ZFC.
By Theorem \ref{ccc_card},
\(\Rand_{\omega_2}\) forces that \(\dot \aleph_1 = \check \aleph_1\) and \(\dot \aleph_2 = \check \aleph_2\).
On the other hand, we have already observed that for all $\alpha < \beta < \omega_2$,
\[
\one \forces_{\Rand_{\omega_2}} \dot r_\alpha \ne \dot r_\beta.
\]
Hence $\Rand_{\omega_2}$ forces that $|\dot \R| \geq \check \aleph_2 = \dot \aleph_2$.
Since the set of formulas which are forced by $\one$ is a consistent theory
extending ZFC and containing $|\R| \geq \aleph_2$,
this establishes that that ZFC cannot prove the Continuum Hypothesis.
The same argument shows that $\Cohen_{\omega_2}$ forces that CH is false;
this was the essence of Cohen's proof \cite{CON_negCH}.

\section[Intersection property]{An intersection property of families of sets of positive measure*}

\label{ameoba:sec}

The purpose of this section is to use the tools which we have developed in order to
prove the following intersection property of sets
of positive measure in $[0,1]$.

\begin{prop}
If $X \subseteq \R$ is uncountable and
$\Seq{B_x : x \in X}$ is an indexed collection of Borel subsets of $[0,1]$,
each having positive measure,
then there is a nonempty set $Y \subseteq X$ such that $Y$ has no isolated points and
such that
$\bigcap \{B_y : y \in Y\}$ has positive measure.
\end{prop}

\begin{proof}
By replacing each $B_x$ with a subset if necessary, we may assume that each $B_x$ is compact.
Similarly, by replacing $X$ with a subset if necessary, we may assume that there is an
$\epsilon > 0$ such that if $x$ is in $X$, then $B_x$ has measure greater than $\epsilon$.
Let $T$ consist of all finite length sequences $\sigma = \Seq{\sigma_i : i < n}$ such that:
\begin{enumerate}
\popcounter

\item $\sigma$ is an increasing sequence of finite subsets of $X$;

\item $\bigcap \{B_x : x \in \sigma_i\}$ has measure greater than $\epsilon$ for all
$i < n$;

\item for each $i < n$, if $x$ is in $\sigma_i$, then there is
a $y$ distinct from $x$ in $\sigma_i$ such that $|x-y| < 1/i$.

\pushcounter
\end{enumerate}
Observe that if $\sigma$ is an infinite sequence all of whose initial parts are in $T$,
then $Y := \bigcup \{\sigma_i : i < \infty\}$ has no isolated points and
$\bigcap \{B_y :y \in Y\}$ has measure at least $\epsilon$.
Conversely, if there is a countable $Y \subseteq X$ with no isolated points and
$\bigcap_{x \in F} B_x$ has measure greater than $\epsilon$ whenever $F \subseteq Y$ is finite,
then $T$ has an infinite path.
Thus by Proposition \ref{WF_abs}, it is sufficient to show that the conclusion of the
proposition is forced by some condition in some forcing.

Consider the Amoeba forcing $\Ameoba_\epsilon$ and let
$\dot Z$ be the $\Qcal$-name for the set
$\{x \in \check X : \check B_x \in \dot G\}$.
Observe that every condition forces that the intersection of every finite subset of
$\{\check B_x : x \in \dot Z\} \subseteq \dot G$ is in $\dot G$ and hence
in $\check \Ameoba_\epsilon$.
By Proposition \ref{generic_prop_K},
there is a $q$ in $\Ameoba_\epsilon$ such that
$q$ forces that $\dot Z$ is uncountable.
By Property \ref{ZFC_forced} of the forcing relation,
$q$ forces that $\dot Z$ contains a countable subset $\dot Y$ with no isolated points.
This finishes the proof.
\end{proof}

\section{The Halpern-L\"auchli theorem*}

The Halpern-L\"auchli Theorem is a Ramsey-theoretic result concerning colorings of
products of finitely branching trees.
Before stating the theorem, we need to first define some terminology.
Recall that a subset $T$ of $\omega^{<\omega}$ is a \emph{tree} if it is closed under initial segments:
whenever $t$ is in $T$ and $s$ is an initial part of $t$, it follows that $s$ is in $T$.
A tree $T \subseteq \omega^{<\omega}$ comes equipped with a natural partial order:
$s \leq t$ if and only if $s$ is an initial part of $t$.
If $T \subseteq \omega^{<\omega}$ is a tree and $l < \omega$, the \emph{$l$th level}
of $T$ consists of all elements of $T$ of length $l$ and is denoted $(T)_l$.

All trees considered in this section will be assumed to be \emph{pruned} without further mention:
every element will have at least one immediate successor.
A tree $T \subseteq \omega^{<\omega}$ is \emph{finitely branching} if every element of $T$ has only finitely many immediate
successors in $T$.
If $S \subseteq T \subseteq \omega^{<\omega}$ are trees and $J \subseteq \omega$ is infinite,
then we say that
$S$ is a \emph{strong subtree of $T$ based on $J$} if whenever
$s$ is in $S$ with length in $J$, every immediate successor 
of $s$ in $T$ is in $S$.
The Halpern-L\"auchli Theorem can now be stated as follows.
 
\begin{thm} \cite{Halpern-Lauchli}
If $\Seq{T_i : i < d}$ is a sequence of finitely branching subtrees of
$\omega^{<\omega}$ and
\[
f : \bigcup_{l = 0}^\infty \prod_{i < d} (T_i)_l \to k
\]
then there exists an infinite set $L \subseteq \omega$ and strong subtrees
$S_i \subseteq T_i$ based on $L$ for each $i < d$ such that
$f$ is constant when restricted to $\bigcup_{l \in L} \prod_{i < d} (S_i)_l$.
\end{thm}

Unlike essentially all other Ramsey-theoretic statements concerning the countably infinite,
the full form of the Halpern-L\"auchli Theorem --- at least at present --- cannot be derived from the
machinery of  semigroup dynamics of spaces of ultrafilters
(see \cite{alg_betaN}, \cite{intro_Ramsey_spaces}).
The special case of the Halpern-L\"auchli Theorem for $n$-ary trees is a consequence
of a form of the Hales-Jewett Theorem, which can be proved using semigroup dynamics
--- see \cite{intro_Ramsey_spaces}.
The proof which is presented in this section is based on forcing and is an inessential modification
of an argument due to Leo Harrington
(see \cite{forcing_appl}).

In order to prove the Halpern-L\"auchli Theorem, we will derive it from the so-called \emph{dense set}
form of the theorem.
If $T \subseteq \omega^{<\omega}$ is a tree and $t$ is in $S$, then a set $D \subseteq T$ is
\emph{$(m,n)$-dense in $T$ above $t$}
if $D \subseteq (T)_n$ and whenever $u$ is in $(T)_m$ with $t \subseteq u$,
there is a $v$ in $D$ such that $u \subseteq v$.
If $t$ is the null string, then we just say that $D$ is $(m,n)$-dense in $T$.

\begin{thm} \label{HLD}
If $\Seq{T_i : i < d}$ is a sequence of finitely branching subtrees of $\omega^{<\omega}$ and
\[
f : \bigcup_{l = 0}^\infty \prod_{i < d} (T_i)_l \to k
\]
then there is an $l$ and a $\bar t$ in $\prod_{i < d} (T_i)_l$
such that for every $m \geq l$ there is an $n \geq m$ and sets
$\Seq{D_i : i < d}$ such that for each $i < d$, $D_i$ is $(m,n)$-dense above
$t_i$ in $T_i$ and such that $f$ is constant on $\prod_{i < d} D_i$.
\end{thm}

The original form of the Halpern-L\"auchli Theorem is an
immediate consequence of the dense set version and the following observation.

\begin{obs}
Let $T \subseteq \omega^{<\omega}$ be a tree and $t$ be an element of $T$.
If $\Seq{D_p : p < \infty }$ is a sequence of subsets of $T$
such that for some increasing sequence
$\Seq{m_p : p < \infty }$, $D_p$ is $(m_p,m_{p+1})$-dense in $T$ above $t$,
then the downward closure of $\bigcup_{p =0}^\infty D_p$ contains a strong subtree of
$T$ which is based on $\{m_p : p < \infty \}$. 
\end{obs}

It is also not difficult to see that, unlike the standard formulation of the
Halpern-L\"auchli Theorem, the special case of Theorem \ref{HLD} in which each
$T_i$ is $2^{<\omega}$ is equivalent to the theorem in its full generality.

Harrington's proof of the Halpern-L\"auchli theorem uses the forcing relation to
reduce the desired Ramsey-theoretic properties of trees to Ramsey-theoretic properties
of cardinals.
In the proof we will need some standard definitions
and facts from combinatorial set theory (see, e.g., \cite[Ch.II]{set_theory:Kunen}).
If $\kappa$ is a regular cardinal, a subset $S$ of $\kappa$ is
\emph{stationary} if it intersects every closed and unbounded subset of $\kappa$.
Clearly every stationary subset of $\kappa$ has cardinality $\kappa$.
Furthermore, if $\mu$ is an infinite cardinal less than $\kappa$, then 
the set of all ordinals in $\kappa$ of cofinality $\mu$ is stationary.
We will need the following property of stationary sets.

\begin{lem}[Pressing Down Lemma; see \cite{set_theory:Kunen}] \label{PDL}
Suppose that $\theta$ is a regular cardinal
and $S \subseteq \theta$ is a stationary set.
If $r:S \to \theta$ satisfies that
$r(\xi) < \xi$ for all $\xi \in S$,
then $r$ is constant on a stationary subset of $S$.
In particular if a stationary subset of $\theta$ is partitioned
into fewer than $\theta$ sets, then
one of the pieces of the partition is stationary.
\end{lem}

We will need the following variant of the $\Delta$-System Lemma.

\begin{lem} \label{base_case}
Suppose that $X$ is a set, $\theta$ is
the successor of a regular cardinal,
and $\{p_\xi : \xi < \theta\}$
is a family of partial functions from $X$
to $2$ such that for every $\xi < \theta$,
$2^{|\dom(p_\xi)|} < \theta$.
Then there exists a cofinal $H \subseteq \theta$ such that
$\bigcup_{\xi \in H} p_\xi$ is a function.
\end{lem}

\begin{proof}
Set $\theta := \kappa^+$.
Observe that by replacing $X$ with the union of the domains of the $p_\xi$'s if necessary,
we may assume that $|X| \leq \theta$ and thus moreover that $X \subseteq \theta$.
For each $\xi < \theta$,
define $a_\xi := \dom (p_\xi) \cap \xi$.
Observe that $|\dom(p_\xi)|$ must be less than $\kappa$ for each $\xi$
and thus $a_\xi$ is a bounded subset of $\xi$ whenever $\cf(\xi) = \kappa$.
Let $E \subseteq \theta$ consist of all $\delta$ such that
if $\xi < \delta$, then $\sup (\dom (p_\xi)) < \delta$.
It is easily checked that $E$ is a closed and unbounded set.
By Lemma \ref{PDL}, there is a stationary $S \subseteq E$ consisting of
ordinals of cofinality $\kappa$ and a $\zeta$ such that
if $\xi$ is in $S$, $\sup a_\xi < \zeta$.
By the pigeonhole principle, there is a stationary set
$H \subseteq S$ and partial function $r$ from $\theta$ to $2$
such that if $\xi$ is in $H$, then $p_\xi \restriction \xi = r$.
Now if $\xi < \eta$ are in $H$, then
$p_\xi \bigcup p_\eta$ is a function.
To see this, suppose that $\alpha$ is in $\dom(p_\xi) \cap \dom(p_\eta)$.
Since $\eta$ is in $E$, it must be that $\alpha < \eta$.
Thus $\alpha$ is in $a_\xi = a_\eta = \dom (r)$ and hence
$p_\xi(\alpha) = r(\alpha) = p_\eta(\alpha)$.
\end{proof}

Next, we will need two closely related Ramsey-theoretic statements which are relatives
of the Erd\H{o}s-Rado Theorem but which have simpler proofs.

\begin{lem} \label{ER_polar}
Suppose that $\Seq{\theta_i : i < d}$ is a sequence of uncountable
regular cardinals satisfying $2^{\theta_i} < \theta_{i+1}$ if $i < d-1$.
If $f:\prod_{i < d} \theta_i \to \omega$,
then there exist cofinal sets $H_i \subseteq \theta_i$ for each $i < d$
such that $f$ is constant on $\prod_{i < d} H_i$.
\end{lem}

\begin{proof}
The proof is by induction on $d$.
If $d$ is given and $\xi < \theta_{d-1}$, fix
$H_i^\xi \subseteq \theta_i$ for each $i < d-1$ such that
$f$ takes the constant value $g(\xi)$ on
\[
\left( \prod_{i < d-1} H_i \right) \times \{\xi\}
\]
(if $d=1$, then the product over the emptyset is the trivial map with
domain $\emptyset$ and this is vacuously true).
By applying the pigeonhole principle and our cardinal arithmetic assumption,
there is an $H_{d-1} \subseteq \theta_{d-1}$ such that $g$ is constant on $H_{d-1}$ and
$H_i^\xi$ does not depend on $\xi$ for $\xi \in H_{d-1}$.
It follows that $f$ is constant when restricted to
$\prod_{i < d} H_i$, where $H_i := H_i^\xi$ for some (equivalently any) $\xi$ in $H_{d-1}$.
\end{proof}

\begin{lem} \label{ER_Delta}
Suppose that $X$ is a set and $\Seq{\theta_i : i < d}$
is a sequence of successors of infinite regular cardinals such that
$2^{\theta_i} < \theta_{i+1}$ if $i < d-1$.
If $\{p_\sigma : \sigma \in \prod_{i < d} \theta_i\}$
is a family of finite partial functions from $X$ into a countable set,
then there are $H_i \subseteq \theta_i$ of cardinality $\theta_i$
such that
\[
\bigcup \{p_\sigma : \sigma \in \prod_{i < d} H_i\}
\]
is a function.
\end{lem}

\begin{proof}
The proof is by induction on $d$.
The case $d=1$ follows from Lemma \ref{base_case}.
Now suppose $\Seq{\theta_i : i \leq d}$ and
$\Seq{p_\sigma : \sigma \in \prod_{i \leq d} \theta_i}$ are given.
For each $\xi$ in $\theta_{d}$, find $\Seq{H_i^\xi : i < d}$ such
that $H^\xi_i \subseteq \theta_i$ and
such that
\[
\bigcup \{p_{\sigma \cat \xi} : \sigma \in \prod_{i < d} H_i\}
\]
is a function, which we will denote by $q_\xi$.
Applying the pigeonhole principle, find
a cofinal $\Gamma \subseteq \theta_{d}$ such that,
for some $\Seq{H_i : i < d}$,
$H_i^\xi = H_i$ if $\xi$ is in $\Gamma$.
Now apply Lemma \ref{base_case} to $\Seq{q_\xi : \xi \in \Gamma}$
to find $H_d \subseteq \Gamma$ of cardinality $\theta_d$ such
that
\[
\bigcup_{\xi \in H_d} q_\xi = \bigcup \{p_\sigma : \sigma \in \prod_{i\leq d} H_i\}
\]
is a function.
\end{proof}

Finally we turn to the task of proving the dense set form of the Halpern-L\"auchli Theorem.

\begin{proof}[Proof of Theorem \ref{HLD}]
Suppose that 
\[
f: \bigcup_{k = 0}^\infty \prod_{i < d} 2^k \to 2
\]
is given and define $\Seq{\theta_i : i < d}$ by $\theta_0 = \aleph_1$
and $\theta_{i+1} = (2^{\theta_i})^{++}$.
Set $\Qcal$ to be the collection of all finite partial functions from $\theta_{d-1}$
into $2^{<\omega}$.
(This is an inessential modification of the forcing $\Cohen_{\theta_{d-1}}$.)
The order on $\Qcal$ is defined by $q \leq p$ if the domain of $q$
contains the domain of $p$ and
$p(\alpha)$ is an initial part $q(\alpha)$ whenever $\alpha$ is in the domain of $p$.
Observe that $\dot r_\xi := \bigcup_{q \in \dot G} q(\check \xi)$ describes an element of $2^{\omega}$.

Applying Property \ref{ZFC_forced} of the forcing relation,
fix a $\Qcal$-name $\dot \Ucal$ for a nonprincipal ultrafilter on $\omega$.
Since $\dot \Ucal$ is forced to be an ultrafilter,
for each $\sigma \in \prod_{i<d} \theta_i$ there is a $\dot e_\sigma$ such that
it is forced that there such that 
\[
\dot U_\sigma =
\{
m \in \omega :
f(\dot r_{\sigma(0)} \restriction m ,\ldots,\dot r_{\sigma(d-1)} \restriction m) =
\dot e_\sigma
\}
\]
is in $\dot \Ucal$.
By Property \ref{decide} of the forcing relation,
there is a $p_\sigma$ which decides $\dot e_\sigma$ to be some $e_\sigma \in \{0,1\}$.
By extending $p_\sigma$ if necessary, we may assume that there is an $l_\sigma$
such that if $\alpha$ is in
the domain of $p_\sigma$, $p_\sigma(\alpha)$ has length $l_\sigma$.
Define
\[
g(\sigma) := \big(e_\sigma,l_\sigma,p_\sigma(\sigma(0)),\ldots,p_\sigma(\sigma(d-1))\big).
\]
By Lemmas \ref{ER_polar} and \ref{ER_Delta}, there are cofinal sets $H_i \subseteq \theta_i$
such that:
\begin{enumerate}
\popcounter

\item
$g$ is constantly $(e,l,t_0,\ldots,t_{d-1})$ on $\prod_{i < d} H_i$ for some $(e,l)$ and
$t_0,\ldots,t_{d-1}$ in $2^l$;

\item
every finite subset of $\{p_\sigma : \sigma \in \prod_{i < d} H_i\}$
has a common lower bound in $\Qcal$.

\pushcounter
\end{enumerate}
Now let $m \geq l$ be given.
For each $i < d$, let $A_i$ be a subset of $H_i$ of cardinality $2^{m-l}$
and fix a bijection between $A_i$ and the set of binary sequences of length $m-l$.
Let $q$ be a condition in $\Qcal$ which is a common lower bound for
\[
\{p_\sigma : \sigma \in \prod_{i < d} A_i\}
\] 
and such that if $\alpha$ is the element of $A_i$ which corresponds
to $u \in 2^{m-l}$ under the bijection,
then $q(\alpha)$ has $t_i \cat u$ as an initial part.
That is, $q$ forces that $\dot r_\alpha \restriction m = t_i \cat u$.

Let $\bar q$ be an extension of $q$ such that for some $n > m$,
$\bar q$ forces that
\[
\check n \in \bigcap \{\dot U_\sigma : \sigma \in \prod_{i < d} A_i\}.
\]
By extending $\bar q$ if necessary, we may assume that for each $i < d$ and
$\alpha$ in $A_i$,
$\bar q(\alpha)$ has length at least $n$.
Finally, set $D_i$ to be the set of all $w \in 2^n$ such that for some $\alpha$ in $A_i$,
$\bar q(\alpha) \restriction n = w$.

We will now show that $D_i$ is $(m,n)$ dense above $t_i$ for each $i < d$ and that
$f \restriction \prod_{i < d} D_i$ is constantly $e$.
To see the former, fix $i < d$ and let $u$ be in $2^{m-l}$ and let $\alpha$ be the corresponding element of
$A_i$.
By our choice of $q$, $q(\alpha) \restriction m = t_i \cat u$
and by our choice of $\bar q$ and
the definition of $D_i$, there is a $w$ in $D_i$ such that
$\bar q(\alpha) \restriction n = w$.
To see the latter, 
suppose $\Seq{w_i : i < d} \in \prod_{i < d} D_i$ and
let $\sigma \in \prod_{i < d} A_i$ be such that
$\bar q(\alpha_i) \restriction n = w_i$.
Clearly $\bar q$ forces that
\[
\Seq{\dot r_{\sigma(i)} \restriction n : i < d} = \Seq{w_i : i < d}.
\]
Furthermore, by the definition of $U_\sigma$ and $n$,
we have that
\[
f(\Seq{\dot r_{\sigma(i)} \restriction n : i < d}) =
f(\Seq{w_i : i < d}) = e.
\]
\end{proof}

\section{Universally Baire sets and absoluteness}

\label{abs:sec}

In this section, we will introduce an abstract notion of regularity for subsets
of complete metric spaces which is useful in proving absoluteness results.
Let $(X,d)$ be a (not necessarily separable) complete metric space.
Recall that the completion of a metric space is taken to be the collection of all
equivalence classes of its Cauchy sequences.
Recall also that if $\Qcal$ is a forcing, then $\dot X$ represents a $\Qcal$-name for
the completion of $\check X$.

In this section we will be interested in interpreting names by filters which 
are not fully generic.
Notice, for instance, that it is possible that $\dot y$ is forced to be equal to $\check x$, even
though there are some (non generic) ultrafilters which interpret $\dot y$ to be different than $x$.
For this reason it is necessary to work with names which have better properties
with respect to arbitrary interpretations.

\begin{defn}
If $\Qcal$ is a forcing, then a \emph{nice $\Qcal$-name for an element of $\dot X$}
is a $\Qcal$-name $\dot x$ such that, for some countable collection of dense subsets $\Dcal$
of $\Qcal$, $\Int{\dot x}{G}$ is a Cauchy sequence in $(X,d)$ whenever
$G$ is $\Dcal$-generic.
\end{defn}

\begin{remark}
For technical reasons we need to make nice $\Qcal$-names for elements of a complete metric space
$\dot X$ to formally be a Cauchy sequence rather than an equivalence class
of a Cauchy sequence, even though
the intent is only to refer to the limit point corresponding to the equivalence class. 
Also, while the completion of a complete metric space is not literally equal to the original
space, there is a canonical isometry between the two and usually there is no need to
distinguish them.
The point in the above definition is that the only meaningful way to define
$\dot X$ is as the name for the completion of $\check X$.
Hence names for elements of $\dot X$ are names for equivalence classes of Cauchy sequences.
When they are interpreted by a sufficiently
generic filter, they will typically result in elements of the completion of $X$, not
in elements of $X$.
\end{remark}

The next lemma shows nice $\Qcal$-names can be used to represent any element of $\dot X$ whenever
$\Qcal$ is a forcing and $X$ is a complete metric space.

\begin{lem}
If $\Seq{\dot x_n : n < \infty}$ is a $\Qcal$-name for a Cauchy sequence in $\check X$, then
there exists a nice $\Qcal$-name $\Seq{\dot y_n : n < \infty}$ for an element of $\dot X$ such that
it is forced that $\Seq{\dot y_n : n < \infty} = \Seq{\dot x_n : n < \infty}$.
\end{lem}

\begin{proof}
Define
\[
\dot y_n := \{(\check y,q) \in X \times \Qcal : q \forces \dot x_n = \check y\} 
\]
and let $D_n$ be the elements of $\Qcal$ which decide both $\dot x_n$ and the least $\dot m$
such that for all $i,j  > \dot m$, $d(\dot x_i,\dot x_j) < 1/n$.
It is readily verified that $\Dcal := \{D_n : n < \infty\}$ witnesses that
$\Seq{\dot y_n : n < \infty}$
is a nice $\Qcal$-name for an element of $\dot X$.
\end{proof}

\begin{defn} (see \cite{UB})
Let $(X,d)$ be a complete metric space.
A subset $A$ of $X$ is \emph{universally Baire}
if whenever $\Qcal$ is a forcing there is a $\Qcal$-name $\dot A$ such that for every
nice $\Qcal$-name $\dot x$ for an element of $\dot X$, there is a countable collection of
dense subsets $\Dcal$ of $\Qcal$ such that:
\begin{enumerate}[\indent a.]

\item $\{q \in \Qcal : q \textrm{ decides } \dot x \in \dot A\}$ is in $\Dcal$;

\item whenever $G$ is a $\Dcal$-generic filter in $\Qcal$,
$\Int{\dot x}{G}$ is in (the completion of) $X$
and $\Int{\dot x}{G}$ is in $A$ if and only
if there is a $q$ in $G$ such that $q \forces \dot x \in \dot A$.

\end{enumerate}
\end{defn}

The following proposition, while easy to establish, is
important in what follows.
\begin{prop}
If $\dot A$ and $\dot B$ are $\Qcal$-names which both witness that
$A$ is universally Baire with respect to $\Qcal$,
then $\one \forces \dot A = \dot B$.
\end{prop}

\begin{proof}
If this were not the case, then there would exist a nice
$\Qcal$-name $\dot x$ for an element of $\dot X$and
a $p$ in $\Qcal$ such that
\[
p \forces \dot x \in (\dot A \symdif \dot B).
\]
Suppose without loss of generality that $p \forces \dot x \in (\dot A \setminus \dot B)$.
If $G$ is a sufficiently generic filter containing $p$, then
$\Int{\dot x}{G}$ will be in $A$ since $p$ is in $G$ and $p$ forces that $\dot x$ is in $\dot A$.
On the other hand, $\Int{\dot x}{G}$ can't be in $A$ since $p$ is in $G$ and
$p$ forces that $\dot x$ is not in $\dot B$,
a contradiction.
\end{proof}

The following is also easy to establish.
The proof is left to the interested reader.

\begin{prop} \label{UB_sigma-alg}
The universally Baire subsets of a complete metric space form a $\sigma$-algebra
which includes the open subsets of $X$.
In particular, every Borel set in a complete metric space is universally Baire.
\end{prop}

Putting this all together, we have the following proposition which will be
used in establishing absoluteness results.
If $\phi(v_1,\ldots,v_n)$ is a logical formula and $x_1,\ldots,x_n$ are sets,
then we say that $\phi(x_1,\ldots,x_n)$ is \emph{generically absolute}
if whenever $\Qcal$ is a forcing and $q$ is in $\Qcal$,
$q \forces \phi(\check x_1,\ldots,\check x_n)$ if and only if
$\phi(x_1,\ldots,x_n)$ is true.

\begin{prop} \label{gen_abs_prop}
The assertion that a given countable Boolean combination of universally
Baire subsets of a complete metric space
is nonempty is generically absolute.
\end{prop}

\section{A property of marker sequences for Bernoulli shifts*}

In this section we will give an example of how homogeneity properties of a forcing
can be put to use.
The goal of the section is to prove a special case of a theorem of Gao, Jackson,
and Seward concerning marker sequences in Bernoulli shift actions. 
Let $\Gamma$ be a countable discrete group acting continuously on a Polish space $X$.
A decreasing sequence of Borel subsets
$\Seq{A_n : n < \infty}$ of
$X$ is a \emph{vanishing marker sequence} for the action if each $A_n$ intersects
every orbit of the action and $\bigcap_{n=0}^\infty A_n = \emptyset$.
Certainly a necessary requirement for such a sequence to exist is that every orbit of
$\Gamma$ is infinite.
In fact this is also a sufficient condition; this is the content of the so-called
\emph{Marker Lemma} (see, e.g., \cite{KechrisMiller}).
The following result of Gao, Jackson, and Seward
grew out of their analysis of the Borel chromatic number
of the free part of the shift graph on $2^{\Z^2}$.

\begin{thm} \cite{grp_color_bernoulli}
Suppose that $\Gamma$ is a countable group,
$k$ is a natural number, and $\Seq{A_n : n < \infty}$ is a vanishing marker sequence
for the free part of the shift action of $\Gamma$ on $k^\Gamma$.
For every increasing sequence $\Seq{F_n : n < \infty}$ of finite sets which covers
$\Gamma$, there is an $x \in k^\Gamma$ such that:
\begin{enumerate}[\indent a.]

\item the closure of the orbit of $x$ is contained in the free part of the action;

\item the closure of the orbit of $x$ is a minimal nonempty closed subset of
$k^\Gamma$ which is invariant under the action;

\item there are infinitely many $n$ such that, for some $g$ in $F_n$,
$g \cdot x$ is in $A_n$.

\end{enumerate}
\end{thm}

Our interest will be primarily in the last clause, although in this generality,
\cite{grp_color_bernoulli} represents the first proof 
of the first two clauses in the theorem.

We will focus our attention on the special case $\Gamma = \Z$ and $k =2$
(the case $\Z^d$ and $k$ arbitrary is just notationally more complicated,
but the full generality of the theorem requires a different argument).
In this context, the above theorem can be rephrased as follows.

\begin{thm} \label{GJS} \cite{grp_color_bernoulli}
Suppose that $\Seq{A_n : n < \infty}$ is a vanishing marker sequence
for the free part of the action of $\Z$ on $2^\Z$ by shift.
For every $f: \N \to \N$ such that $\lim_n f(n) = \infty$,
there exists an $x$ in $2^\Z$ such that:
\begin{enumerate}[\indent a.]

\item the closure of the orbit of $x$ is contained in the free part of the action;

\item the closure of the orbit of $x$ is a minimal nonempty closed subset
of $2^\Gamma$ which
is invariant under the action;

\item there are infinitely many $n$ such that,
for some $i \in [-f(n),f(n)]$, $x + i$ is in $A_n$.

\end{enumerate}
\end{thm}


It will be useful to have a more combinatorial way of
formulating the first two conclusions of the theorem.
They are provided by the following lemmas.

\begin{lem} (see \cite{grp_color_bernoulli}) \label{free_char}
If $x$ is in $2^\Z$, then the closure of the orbit of $x$ is contained
in the free part of the action exactly
when the following condition holds:
for all $a \in \Z \setminus \{0\}$, there exists a finite interval
$B \subseteq \Z$ such that for all $c \in \Z$ there is a $b \in B$
with
\[
x(a + b + c) \ne x(b+c).
\]
\end{lem}

\begin{lem} (see \cite{grp_color_bernoulli}) \label{min_char}
If $x$ is in $2^\Z$, then the closure of the orbit of $x$
is a minimal closed invariant subset of $2^\Z$ if and only if
the following condition holds:
for every finite interval $A \subseteq \Z$, there is a finite interval
$B \subseteq \Z$ such that for all $c \in \Z$ there is a $b \in B$ such that
for all $a \in A$
\[
x(a + c) = x(a+b).
\]
\end{lem}

Define $\Qcal$ to consist of all finite partial functions $q : \Z \to 2$ such that
the domain of $q$ is an interval of integers, denoted $I_q$.
If $n$ is in $\Z$ and $q$ is in $\Qcal$, then the \emph{translate of $q$ by $n$}
is denoted
$q + n$ and is defined by $(q+n)(i) = q(i-n)$
(with $q+n$ having domain $\{i \in \Z : i-n \in I_q\}$).
If $q$ is in $\Qcal$, define $\bar q$ to be the bitwise complement of $q$:
$\bar q(i) := 1-q(i)$.
A disjoint union of conditions which is again a condition
will be referred to as a \emph{concatenation}.

The order on $\Qcal$ is defined by $q \leq p$ if $q = p$
or else $q$ is a concatenation of a set $S$ of conditions, each of which is a translate
of $p$ or $\bar p$ and such that $p$ and a translate of $\bar p$ are in $S$.
If $G$ is a filter in $\Qcal$ such that $\bigcup G$ is a total function $x$
from $\Z$ to $2$, then
it is straightforward to check that $x$
satisfies the conclusion of Lemma \ref{min_char} and thus that the closure
of the orbit of $x$ is a minimal invariant closed subset of $2^{\Z}$.

Under a mild genericity assumption on $G$,
the closure of the orbit of $x$ will be contained in the free part of the action.
For each $p \ne \one$ in $\Qcal$ and $i < \infty$,
define $D_{p,i}$ to be the set of all $q$ in $\Qcal$ such that either $q$
is incompatible with $p$ or else $q \leq p$, $m+i$ is in the domain of $q$, and
$q(m+i) \ne q(m)$, where $m := \max (I_p)$.

\begin{lem}
Each $D_{p,i}$ is a dense subset of $\Qcal$ and if $G \subseteq \Qcal$
is a filter which intersects $D_{p,i}$ for each 
$p \in \Qcal \setminus \{\one\}$ and $i < \infty$,
then $x := \bigcup G$ satisfies that the closure of the orbit of $x$
is contained in the free part of the action.
\end{lem}


Thus we have shown that $\one$ forces that $\dot x = \bigcup G$ satisfies that
the closure of the orbit of $\dot x$ is minimal and contained in the free part of the Bernoulli shift.
The following proposition, when combined with Proposition \ref{gen_abs_prop},
implies Theorem \ref{GJS}.

\begin{prop}
Suppose that $\Seq{A_n : n < \infty}$ is a vanishing sequence of
markers for the free part of
the Bernoulli shift $\Z \act 2^{\Z}$ and that
$f:\N \to \N$ is a function such that $\lim_n f(n) = \infty$.
Every condition in $\Qcal$ forces that for every $m$ there is an $n \geq m$
such that $\dot x + k \in \dot A_n$ for some $k$ with $-f(n) \leq k \leq f(n)$.
\end{prop}

\begin{proof}
Suppose for contradiction that this is not the case.
Then there is a $p$ in $\Qcal$ such that $p$ forces:
there is an $m$ such that for every $n \geq m$,
if $-f(n) \leq k \leq f(n)$ then $\dot x + k \not \in \dot A_m$.
By replacing $p$ with a stronger condition if necessary,
we may assume that there is an $m$ such that $p$ forces that for every $n \geq m$,
if $-f(n) \leq k \leq f(n)$ then $\dot x + k \not \in \dot A_m$.
Let $q := p \cup (\bar p + l)$ where $l$ is the length of $I_p$.
Observe that $q \leq p$ and that if $r \leq q$ and $i \in \Z$, then there is a $j$ with
$0 \leq j < 2l$ such that $r - i+j$ is compatible with $p$;
simply choose $j$ such that $i-j$ is a multiple of $2l$.

Let $n \geq m$ be such that $f(n)$ is greater than $2l$ and find a $r \leq q$
and an $i \in \Z$ such that $r$ forces that $\dot x + i$ is in $\dot A_n$.
This is possible since it is forced that $\dot A_n$
meets every orbit
(strictly speaking, we are appealing to Proposition \ref{gen_abs_prop} here).
Now, let $j < 2|p|$ be such that $r$ is compatible with $p+i+j$.

We now have that $r-i+j$ forces that $\dot x + j$ is in $\dot A_n$.
This follows from the fact that
\[
\one \forces \dot x - i + j \textrm{ is generic over the ground model}.
\]
(This follows from the observation that if $E \subseteq \Qcal$ is exhaustive, then
so is any translate of $E$.)
Recall that $r-i+j$ is compatible with $p$ and let
$s$ be a common lower bound for $r-i+j$ and $p$.
It follows that $s$ forces that both $\dot x + j \not \in \dot A_n$ and
$\dot x + j \in \dot A_n$, a contradiction. 
\end{proof}

\section{Todorcevic's absoluteness theorem for Rosenthal compacta}

\label{RC_absolute:sec}

We will now use the results of Section \ref{abs:sec}
to prove Todorcevic's absoluteness theorem for Rosenthal compacta.
Fix a Polish space $X$.
Recall that a real valued function defined on $X$ is \emph{Baire class 1}
if it is the limit of a pointwise convergent sequence of continuous functions.
Baire characterized functions $f$ which are \emph{not} Baire class 1 as those for which there exist rational
numbers $p < q$ and nonempty sets $D_0,D_1 \subseteq X$ such that the closures of $D_0$ and $D_1$ coincide, have no isolated
points, and
\[
\sup_{x \in D_0} f(x) \leq p < q \leq \inf_{x \in D_1} f(x)
\]
(see \cite{Baire_char}).
The collection of all Baire class 1 functions on a Polish space $X$
is denoted $BC_1(X)$ and is equipped with the topology of
pointwise convergence.

A compact topological space which is homeomorphic to a subspace of $BC_1(X)$ is said
to be a \emph{Rosenthal compactum}.
This class includes all compact metric spaces and is closed under taking closed subspaces and
countable products.
The following are typical nonmetrizable examples.

\begin{example}[Helly's space; the double arrow]
The collection of all
nondecreasing functions from $[0,1]$ to $[0,1]$ is
known as Helly's space.
It is convex as a subset of $\R^{[0,1]}$.
The extreme points of this set are the characteristic functions of the intervals
$(r,1]$ and $[r,1]$.
This subspace is homeomorphic to the so-called \emph{double arrow space}:
the set $[0,1] \times 2$ equipped with the order topology
from the lexicographic order.  
\end{example}

\begin{example}[one point compactification]
The constant $0$ function, together with the functions $\delta_r : [0,1] \to \R$ defined by
\[
\delta_r(t) :=
\begin{cases} 1 & \textrm{ if } t = r \\
0 & \textrm{ otherwise.}
\end{cases}
\]
This is homeomorphic to the one point compactification of a discrete
set of cardinality $2^{\aleph_0}$.
\end{example}

Rosenthal compacta enjoy a number of strong properties similar to those of compact metric spaces.
One which will play an important role below is \emph{countable tightness}:
a topological space $Z$ is \emph{countably tight}
if whenever $a$ is in the closure of $A \subseteq Z$,
there is a countable $A_0 \subseteq A$ such that $a$ is in the closure of $A_0$.

\begin{thm} \cite{iso_th_Banach} \label{RC_ctbly_tight}
Rosenthal compacta are countably tight.
\end{thm}

In \cite{Rosenthal_cpt}, Todorcevic derived a number of properties of Rosenthal compacta by
showing that there is a natural way to reinterpret such spaces as Rosenthal compacta in
generic extensions.
This result is in fact a fairly routine consequence of the machinery
which was developed in Section \ref{abs:sec} above.

First, we must verify that elements of $\BC_1(X)$ extend to elements of $\BC_1(X)$
in the generic extension.

\begin{lem} \cite{Rosenthal_cpt} \label{BC_ext}
Suppose that $\Seq{ f_n : n < \infty}$ is a sequence of continuous
functions on a Polish space $X$.
The assertion that $\Seq{ f_n : n < \infty}$ converges pointwise is generically absolute.
Furthermore, if $\Seq{ f_n : n < \infty}$ and $\Seq{ g_n : n < \infty}$ are sequences
of continuous functions on $X$, the assertion that
$f_n-g_n \to 0$ pointwise on $X$ is generically absolute.
\end{lem}

\begin{proof}
Let $\Seq{ f_n : n < \infty}$ be sequence of continuous functions.
Observe that
\[
\bigcup_{\epsilon > 0} \bigcap_{n=0}^\infty \bigcup_{i,j \geq n} \{x \in X : |f_i(x) - f_j(x)| > \epsilon\}
\]
specifies a countable Boolean combination of open subsets of $X$ which is empty
if and only if $\Seq{ f_n : n < \infty}$ converges pointwise.
Thus the assertion that $\Seq{ f_n : n < \infty}$ converges pointwise is
generically absolute by Proposition \ref{gen_abs_prop}.
The second conclusion is verified in a similar manner.
\end{proof}

Now suppose that $\Qcal$ is a forcing, $X$ is a Polish space,
and $f$ is in $\BC_1(X)$.
By Lemma \ref{BC_ext}, $\Qcal$ forces that there is a unique element of $\BC_1(\dot X)$ which
extends $\check f$; fix a $\Qcal$-name $\dot f$ for this extension.
If $K \subseteq \BC_1(X)$ is a Rosenthal compactum,
then $\dot K$ is a $\Qcal$-name for the closure of the set of extensions of elements of
$\check K$ to $\dot X$.
(Specifically, it is a $\Qcal$-name for the closure of
$\{(\dot f,\one) : f \in K\}$ in $\dot \R^{\dot X}$.)
Todorcevic's absoluteness theorem can now be stated as follows.

\begin{thm} \cite{Rosenthal_cpt} \label{RC_abs}
Suppose that $X$ is a Polish space and $\Fcal$ is a family of Baire class 1 functions.
The assertion that every accumulation point of $\Fcal$ is Baire class 1 is generically absolute.
\end{thm}

\begin{proof}
Let $X$ and $\Fcal$ be fixed and let $\Qcal$ be a forcing.
It is sufficient to show that the assertion that $\Fcal$ has a pointwise accumulation point which
is not in $\BC_1(X)$ is equivalent to a certain countable Boolean combination of open sets in a completely metrizable
space being nonempty.
Let $Z$ be the set of all sequences
\[
\Seq{\Seq{f_{k,i} : i < \infty} : k < \infty}
\]
such that, for each $k$, $\Seq{f_{k,i} : i < \infty}$ is a sequence of continuous functions
which converges pointwise to an element of $\Fcal$.
We will regard $Z$ as being a product of discrete spaces, noting that with this topology, $Z$ is
completely metrizable.

Observe that if $g$ is a limit point of $\Fcal$ which is not in $\BC_1(X)$, then by Baire's characterization,
there are rational numbers $p < q$, sets $A := \{a_k : k < \infty\}$, $B := \{b_k : k < \infty\}$,
and $\{f_k : k < \infty\} \subseteq \Fcal$ such that:
\begin{enumerate}
\popcounter

\item $A$ and $B$ are contained in $X$, have no isolated points, and have the same closures;

\item if $k < l$, then $f_l(a_k) < p < q < f_l(b_k)$.

\pushcounter
\end{enumerate}
Moreover, one can select sequences $\Seq{f_{k,i} : i < \infty}$ of continuous functions such that
$f_{k,i} \to f_k$ pointwise for each $k$.
Thus we have that for every $k < l$ there is an $n$ such that if $n < j$,
then $f_{l,j} (a_k) < p < q < f_{l,j}(b_k)$.
It follows that there exist
\[
(\Seq{a_k : k < \infty} , \Seq{b_k : k < \infty} , \Seq{\Seq{f_{k,i} : i < \infty} : k < \infty})
\]
in $X^\omega \times X^\omega \times Z$
specifying objects with the above properties if and only if $\Fcal$ has an accumulation point outside
of $\BC_1(X)$.
Notice however, that these properties define a countable Boolean combination of open subsets of
$X^\omega \times X^\omega \times Z$ and therefore
the theorem follows from Proposition \ref{gen_abs_prop}.
\end{proof}

\section{$\sigma$-closed forcings}

\label{sigma-closed:sec}

There are two basic aspects of a forcing which are
of fundamental importance in understanding its properties:
how large are its families of pairwise incompatible elements and how frequently do directed families have lower bounds.
Properties of the former type are often referred to loosely as \emph{chain conditions};
we have already seen the most important of these in Section \ref{ccc:sec}.
Properties of the latter type are known as \emph{closure properties} of a forcing.
In this section, we will discuss the simplest and most important example of a closure property.

\begin{defn}[$\sigma$-closed]
A forcing $\Qcal$ is \emph{$\sigma$-closed} if whenever $\Seq{q_n : n < \infty}$ is
a $\leq$-decreasing sequence of elements of $\Qcal$, there is a $\bar q$ in $\Qcal$
such that $\bar q \leq q_n$ for all $n$.
\end{defn}

It is perhaps worth remarking that any forcing which is $\sigma$-closed and atomless
(i.e. every element has two incompatible extensions), necessarily has an antichain of 
cardinality of the continuum and so in particular is not c.c.c..
Like c.c.c. forcings, however, $\sigma$-closed forcings also preserve uncountability, 
although for a quite different reason.

\begin{prop}
Suppose that $\Qcal$ is a $\sigma$-closed forcing.
If $\dot f$ is a $\Qcal$-name and $p \in \Qcal$ forces that 
$\dot f$ is a function with domain $\omega$, then there is a $q \leq p$ and
a function $g$ such that
$q \forces \dot f = \check g$.
In particular $\one \forces \dot \aleph_1 = \check \aleph_1$ and
$\one \forces \dot \R = \check \R$.
\end{prop}

\begin{proof}
Let $p$ and $\dot f$ be given as in the statement of the proposition.
By repeatedly appealing to Property \ref{decide},
recursively construct a sequence of conditions $\Seq{p_n : n < \infty}$ and
values $g(n)$ of a function $g$ defined on $\omega$
such that for all $n$, $p_{n+1} \leq p_n \leq p$ and 
\[
p_n \forces \dot f(\check n) = \check g(\check n).
\]
Since $\Qcal$ is $\sigma$-closed, there is a $q$ in $\Qcal$ such that
$q \leq p_n$ for all $n$.
Thus by Proposition \ref{check_quant}, it follows that
\[
q \forces \forall n\ (\dot f(n) = \check g(n)).
\]
\end{proof}

We will now consider some examples.
The first forcing provides a means for forcing the Continuum Hypothesis over a given
model of set theory, complementing the discussion at the end of Section \ref{ccc:sec}.

\begin{example}
Let $\Qcal$ denote the collection of all countable partial functions from $\omega_1$ to 
$\R$, ordered by extension.
Let $\dot g$ be the $\Qcal$-name for the union of the generic filter.
It is easily verified that $\Qcal$ forces that $\dot g$ is defined on all of $\check \omega_1$
and maps $\check \omega_1$ onto $\check \R$.
Furthermore, if $\Seq{q_n : n < \infty}$ is a descending sequence of conditions, then
$\bigcup_{n=0}^\infty q_n$ is a condition: it is a function and its domain is countable,
being a countable union of countable sets.
Thus $\Qcal$ is $\sigma$-closed and hence forces that
$\dot \R = \check \R$ and $\dot \aleph_1 = \check \aleph_1$.
Hence $\Qcal$ forces that $|\dot \R| = \dot \aleph_1$
(i.e. that the Continuum Hypothesis is true).
\end{example}

\begin{example}
Consider the forcing $([\omega]^{\omega},\subset)$.
This forcing is neither separative nor $\sigma$-closed.
The separative quotient is obtained by identifying sets
$a$ and $b$ which have a finite symmetric difference.
If we define $a \subseteq^* b$ to mean that $a \setminus b$ is finite,
then $\subseteq^*$ induces the order on the separative quotient.
If $\Seq{A_n : n < \infty}$ is a $\subseteq^*$-decreasing sequence of infinite
subsets of $\omega$,
let $n_k$ be the least element of $\bigcap_{i \leq k} A_i$ which is greater than $n_i$ for
each $i < k$.
Notice that $B := \{n_i : i < \infty\}$ is an infinite set and that
$\{n_i : i \geq k\}$ is a subset of $A_k$.
Thus $B \subseteq^* A_k$ for all $k$.
This shows that the separative quotient is $\sigma$-closed.
Notice that, by Ramsey's theorem, if $f:[\omega]^{d} \to 2$, then
\[
\{q \in [\omega]^{\omega} : f \restriction [q]^{d} \textrm{ is constant}\}
\]
is dense in $[\omega]^{\omega}$ (here $[A]^{d}$ denotes the $d$-element subsets of $A$).
Since the separative quotient of $[\omega]^{\omega}$ is $\sigma$-closed, 
forcing with it does not add new subsets of $\omega$.
Thus it forces that $\dot G$ is a \emph{Ramsey ultrafilter on $\omega$}:
if $f:[\omega]^{d} \to 2$ is a coloring of the $d$-element subsets of $\omega$,
there is an $H$ in the ultrafilter
such that $f$ is constant on the $d$-element subsets of $H$.
Kunen has shown, on the other hand, that whenever $\theta > 2^{\aleph_0}$,  $\Rand_\theta$ forces that there does not
exist a Ramsey ultrafilter on $\omega$ \cite{points_betaN}.
(Kunen actually proved this in the special case in which the ground model satisfies the Continuum Hypothesis.
The general case follows by an absoluteness argument --- forcing with the poset $\Qcal$
of the previous example does not change the truth of ``$\Rcal_{\theta}$ forces that there are no Ramsey ultrafilters on $\omega$''.)
\end{example}

We are now in a position to derive another property of Rosenthal compacta.
The proof below is a reproduction of Todorcevic's proof in \cite{Rosenthal_cpt};
the result itself was originally proved by Bourgain \cite{rem_cpt_BC1} using classical methods.

\begin{thm} \label{RC_point_1st_ctbl}
If $K$ is a Rosenthal compactum,
then $K$ contains a dense set of points with a countable neighborhood base.
\end{thm}

\begin{proof}
Observe that it is sufficient to show that every Rosenthal compactum 
contains a point with a countable base.
Recall the following result of \v{C}ech and Po\v{s}pisil \cite{Gdelta_point}:
if $K$ is a compact topological space of cardinality at most $\aleph_1$, then
$K$ contains a point with a countable neighborhood base.
Let $\Qcal$ be the forcing from the previous example.
We have seen that $\Qcal$ forces that $|\dot \R| = \dot \aleph_1$ and hence
that the collection of all real valued Borel functions on a given Polish space
has cardinality $\aleph_1$.
In particular, $\Qcal$ forces that
any Rosenthal compactum has cardinality $\aleph_1$.

Now, let $K$ be a Rosenthal compactum consisting of Baire class 1 functions on some Polish space $X$.
By Theorem \ref{RC_abs}, $\Qcal$ forces that the closure of $\check K$ inside of $\R^X$ still consists
only of Baire class 1 functions.
Since $\Qcal$ is $\sigma$-closed, it follows that $\one$ forces that $\check K$ is closed and hence a compact space of
cardinality $\aleph_1$.
Therefore by the \v{C}ech-Po\v{s}pisil Theorem, there are $\Qcal$-names $\dot g$ and $\dot U_n$ for each $n$ such that
$\one$ forces that $\dot g$ is an element of $\check K$ and that $\{\dot U_n : n < \infty\}$
is a countable neighborhood base for $\dot g$ consisting of basic open sets.
Since $\Qcal$ is $\sigma$-closed, there is a $q$ in $\Qcal$ which decides $\dot g$ to be some $f$ and
$\dot U_n$ to be some $V_n$ for each $n$.
It follows that $\{V_n: n < \infty\}$ is a countable neighborhood base.
\end{proof}

\section{Mathias reals and a theorem of Galvin and Prikry}

\label{GP:sec}

In this section we will give a forcing proof of the Galvin-Prikry Theorem,
which is an infinite dimensional form of Ramsey's Theorem:

\begin{thm} \cite{Galvin-Prikry}
If $\Xcal \subseteq [\omega]^\omega$ is Borel, then there is an $H \in [\omega]^\omega$ such
that either $[H]^\omega \subseteq \Xcal$ or else $[H]^\omega \cap \Xcal = \emptyset$.
\end{thm}

Recall that Mathias forcing \(\Mathias\) consists of all pairs \(p = (a_p,A_p)\) such that
\(A_p\) is an infinite subset of \(\omega\) and \(a_p\) is a finite initial segment of \(A_p\).
The order on \(\Mathias\) is such that \(q\) extends \(p\) if \(a_p\) is an initial part of \(a_q\) and
\(A_q \subseteq A_p\).
A \emph{Mathias real} is a subset \(X\) of \(\omega\) such that
\[
G_X:=\{p \in \Mathias : a_p \subseteq X \subseteq A_p\}
\]
is a generic filter.
If \(\Dcal\) is a collection of subsets of \(\Mathias\), then we say that \(X\) is \(\Dcal\)-generic
if \(G_X\) is \(\Dcal\)-generic.
If \(D \subseteq \Mathias\), we will say that \(X\) is \(D\)-generic if it is \(\{D\}\)-generic.
We say that \(D \subseteq \Mathias\) is \emph{dense above $n$}
if whenever \(p \in \Mathias\) and \(n \leq \min (a_p)\), \(p\) has an extension in \(D\).

\begin{lem} \label{singleton-generic}
Suppose that \(D \subseteq \Mathias\) is dense above \(n\).
There is a dense set of \(H\) in \(([\omega]^\omega, \subseteq)\)
such that any infinite subset of \(H\) is
\(D\)-generic.
\end{lem}

\begin{proof}
Let \(D\) and \(n\) be given as in the statement of the lemma and let \(A \subseteq \omega\) be arbitrary
with \(n < \min (A)\).
Construct a sequence of infinite subsets \(H_k \subseteq A\) for each \(k\) such that,
setting \(n_k := \min (H_k)\):
\begin{enumerate}
\popcounter

\item \(H_0 := A\) and \(H_{k+1} \subseteq H_k\);

\item \(n_k < n_{k+1}\);

\item for each \(x \subseteq \{n_i : i < k\}\) either there is a \(p \in D\) such that
\(a_p = x\) and \(H_k \subseteq A_p\) or else whenever \(p \in D\) with \(a_p = x\),
\(A_p \cap H_k\) is finite.

\pushcounter
\end{enumerate}
Define \(B :=\{n_k : k < \infty\}\) and set 
\[
\Fcal :=\{x \in [B]^{<\omega} : \exists p \in D ((a_p = x) \mand (A_p \subseteq B ))\}.
\]
By Theorem \ref{GNW}, there is an infinite \(H \subseteq B\) such that either \(H\) has
no subset in \(\Fcal\) or else every infinite subset of \(H\) has an initial segment in \(\Fcal\).
Since \((\emptyset,H)\) is in \(\Mathias\), it has an extension \(p\) in \(D\).
Since \(a_p \subseteq A_p \subseteq H \subseteq B\), \(a_p\) is a subset of \(H\) in \(\Fcal\).
Thus every infinite subset of \(H\) has an initial segment in \(\Fcal\) and \(n \leq \min (H)\).

Now let \(X\) be an infinite subset of \(H\) and \(x\) is an initial part of \(X\) in \(\Fcal\).
Let \(k\) be such that \(x \subseteq \{n_i : i < k\}\) and let \(p \in D\) be such that 
\(x= a_p \subseteq A_p \subseteq B\).
Observe that in particular \(A_p \cap H_k\) is infinite. 
Thus by our construction, \(H_k \subseteq A_q\) for some \(q \in D\) such that
\(a_q = x\).
Now we have that \(a_q \subseteq X \subseteq A_q\) and therefore that
\(X\) is \(D\)-generic.
\end{proof}

\begin{prop} \label{countably-generic}
Suppose that \(\Dcal\) is a countable collection of dense subsets of \(\Mathias\).
For every \(x \in [\omega]^{<\omega}\) there is a dense set of \(H\) in
\([\omega]^\omega\) such that if \(X \subseteq H\) is infinite then \(x \cup X\) is
\(\Dcal\)-generic.
\end{prop}

\begin{proof}
Let \(\Dcal\) and \(x\) be given as in the statement of the proposition.
Fix an enumeration \(\{D_k : k < \infty\}\) of \(\Dcal\) and let \(A \in [\omega]^\omega\) be arbitrary with
\(\max(x) < \min (A)\).
If \(y\) is a finite set and \(k < \infty\), define
\[
D_{k,y} := \{p \in \Mathias : (\max (y) < \min (a_p)) \mand ((y \cup a_p, y \cup A_p) \in D_k) \}.
\]
Observe that \(D_{k,y}\) is dense above \(\max (y) + 1\) and if \(X \subseteq \omega\) with
\(\max (y) < \min (X)\), then \(y \cup X\) is \(D_k\)-generic if \(X\) is \(D_{k,y}\)-generic.

Using Lemma \ref{singleton-generic}, construct infinite sets \(\{H_k : k < \infty\}\) so that:
\begin{enumerate}
\popcounter

\item \(H_0 :=A\) and \(H_{k+1} \subseteq H_k\);

\item setting \(n_k := \min (H_k)\), we have \(n_k < n_{k+1}\);

\item any infinite subset of \(H_k\) is \(D_{j,y}\)-generic whenever \(j < k\) and 
\(x \subseteq y \subseteq x \cup \{n_i : i < k\}\).

\pushcounter
\end{enumerate}
Define \(H:=\{n_k : k < \infty\}\) and suppose that \(X\) is an infinite subset of \(H\).
Let \(k\) be given and set \(y:= x \cup (X \cap \{n_i : i < k\})\).
Since \(X \cap H_k = X \setminus x\) is \(D_{k,y}\)-generic, 
\(x \cup X\) is \(D_k\)-generic.
Thus \(H\) satisfies the conclusion of the proposition.
\end{proof}

If \(p,q \in \Mathias\), then we say that \(q\) is a \emph{pure extension} of \(p\) if
\(q \leq p\) and \(a_p = a_q\).
The following proposition is central to the analysis of \(\Mathias\) and related posets.

\begin{prop} \label{pure_decision}
If \(\phi\) is a formula in the forcing language and \(p \in \Mathias\),
\(p\) has a pure extension which decides \(\phi\).
\end{prop}

\begin{proof}
Let \(p\) and \(\phi\) be given as in the statement of the proposition.
Define \(D\) to be the set of all conditions in \(\Mathias\) which decide \(\phi\), noting that \(D\) is
dense.
By Proposition \ref{countably-generic}, there is an infinite \(B \subseteq A_p\) such that
if \(X \subseteq B\) is infinite, then
\(a_p \cup X\) is \(D\)-generic.
Set 
\[
\Fcal := \{x \in [B]^{<\omega} : \exists p \in \Mathias  ((p \forces \phi) \mand (a_p = x) \mand (A_p \subseteq B))\}.
\]
By Theorem \ref{GNW} there is an infinite \(H \subseteq B\) such that either \(H\) has no subset in
\(\Fcal\) or else every infinite subset of \(H\) has an initial part in \(\Fcal\).
If the first conclusion is true,
then \((a_p,H)\) forces \(\neg \phi\).
If the second conclusion is true, then \((a_p,H)\) forces \(\phi\).
\end{proof}

Since every Borel set is universally Baire by Proposition \ref{UB_sigma-alg},
the next theorem implies the Galvin-Prikry Theorem \cite{Galvin-Prikry}.

\begin{thm}
If \(\Xcal \subseteq [\omega]^\omega\) is universally Baire,
then there is an \(H \in [\omega]^\omega\) such that either
\([H]^\omega \subseteq \Xcal\) or else \([H]^\omega \cap \Xcal = \emptyset\).
\end{thm}

\begin{proof}
Since \(\Xcal\) is universally Baire, there is an \(\Mathias\)-name
\(\dot \Xcal\), countably many dense sets \(\Dcal\), and a name \(\dot X\) for
the Mathias real such that if \(G\) is a \(\Dcal\)-generic filter, then
\(\dot X(G) \in \Xcal\) if and only if there is a \(p \in G\) such that
\(p \forces \dot X \in \dot \Xcal\) if and only if there is no \(p \in G\) such that
\(p \forces \dot X \not \in \dot \Xcal\).
By Proposition \ref{pure_decision}, there is condition of the form \((\emptyset,A)\) which decides
\(\dot X \in \dot \Xcal\).
By Proposition \ref{countably-generic},
there is an \(H \in [A]^\omega\) such that every infinite subset of \(H\) is \(\Dcal\)-generic.
It follows that \(H\) satisfies the conclusion of the theorem.
\end{proof}

\section{When compacta have dense metrizable subspaces*}

Suppose that $K$ is a compact Hausdorff space.
In this section we will reformulate the question of when $K$
contains a dense metrizable subspace in terms of the language of forcing.
Recall that every compact Hausdorff space is homeomorphic to a closed subspace of $[0,1]^I$ for
some index set $I$.
In this section, when we reinterpret $K$ in a generic extension, we will take $\dot K$ to be the name
for the closure of $\check K$ in $[0,1]^{\check I}$.

Recall that a \emph{regular pair} in $K$ is a pair $(F,G)$ such that $F$ and $G$ are disjoint
closed $G_\delta$ subsets of $K$.
If $\Xi$ is an ordered set and $\Seq{(F_\xi,G_\xi) : \xi \in \Xi}$ is a sequence of
regular pairs, then we say that $\Seq{(F_\xi,G_\xi) : \xi \in \Xi}$ is a \emph{free sequence}
if whenever $A,B \subseteq \Xi$ are finite and satisfy $\max (A) < \min (B)$, it follows that
\[
\bigcap_{\xi \in A} G_\xi \cap \bigcap_{\xi \in B} F_\xi \ne \emptyset.
\]
Recall also that a collection $\Bcal$ of nonempty open subsets of
$K$ is a \emph{$\pi$-base} if every nonempty open set in $K$ contains an element of $\Bcal$.
 
We note the following result of Todorcevic.
\begin{thm} \cite{free_sequences} \label{pi_base}
If $K$ is any compact Hausdorff space, there is a sequence $\Seq{(F_\xi,G_\xi) : \xi \in \Pi}$
of regular pairs in $K$ such that
$\{\mathrm{int} (G_\xi) : \xi \in \Pi\}$ forms a $\pi$-base for $K$ of minimum cardinality and such that
whenever $\Xi \subseteq \Pi$ satisfies that $\{G_\xi : \xi \in \Xi\}$ has the finite intersection
property, $\Seq{(F_\xi,G_\xi) : \xi \in \Xi}$ is a free sequence.
\end{thm}

The following result is implicit in \cite{Rosenthal_cpt} and is a key component in
Todorcevic's proof that every Rosenthal compactum contains a dense metrizable subspace.
Let $\Qcal_K$ denote the forcing consisting of all nonempty 
open subsets of $K$ ordered so that $q < p$ means that the closure of $q$ is contained in $p$.
We will let $\dot x_G$ denote the $\Qcal_K$-name for the unique element of the intersection of $\dot G$,
when regarded as a collection of open sets.

\begin{thm} \label{sigma_disj_pi-base_char}
Suppose that $K$ is a compact Hausdorff space and $\Seq{(F_\xi,G_\xi) : \xi \in \Pi}$ is a sequence satisfying
the conclusion of Theorem \ref{pi_base}.
The following are equivalent:
\begin{enumerate}[\indent a.]

\item \label{sigma_disj}
$K$ has a $\sigma$-disjoint $\pi$-base.

\item \label{gen_1st_ctbl}
$\Qcal_K$ forces that $\dot x_G$ has a countable neighborhood base.

\item \label{coll_Pi}
$\Qcal_K$ forces that
$|\{\xi \in \check \Pi : \check G_\xi \in \dot G\} | \leq \aleph_0$.

\end{enumerate}
\end{thm}

\begin{proof}
To see that (\ref{sigma_disj}) implies (\ref{gen_1st_ctbl}), first observe that if $\Ucal$
is a $\pi$-base for the topology on $K$, then $\Ucal$ is dense 
as a subset of $\Qcal_K$.
Hence $\Qcal_K$ forces
that $\check \Ucal \cap \dot G$ generates $G$.
Also, if $\Ocal$ is a pairwise disjoint family of open sets, then it is forced
that $|\check \Ocal \cap \dot G| \leq 1$.
Hence if $\Ucal$ is a $\sigma$-disjoint $\pi$-base,
then it is forced that $\check \Ucal \cap \dot G$ is a countable neighborhood base of $\dot x_G$.

The equivalence between (\ref{gen_1st_ctbl}) and (\ref{coll_Pi}) follows from the fact that
\[
\{G_\xi : (\xi \in \Pi) \mand (\dot x_G \in G_\xi)\}
\]
is forced to be a neighborhood base for $\dot x_G$ and that
\[
\{(F_\xi,G_\xi) : (\xi \in \Pi) \mand (\dot x_G \in G_\xi)\}
\]
is a free sequence and hence no smaller neighborhood base can suffice.

Finally, to see that (\ref{coll_Pi}) implies (\ref{sigma_disj}),
suppose that every condition forces that
$|\{\xi \in \check \Pi : \check G_\xi \in \dot G\} | \leq \aleph_0$.
Let $\Seq{\dot \xi_n : n < \infty}$ be a sequence of $\Qcal_K$-names such that
every condition of $\Qcal_K$ forces that
\[
\{\dot \xi_n : n < \infty\} = \{\xi \in \check \Pi : \dot x_G \in \check G_\xi\}.
\]
Let $\Ocal_n$ be a maximal antichain in $\Qcal_K$ such that elements of $\Ocal_n$ decide
$\dot \xi_n$ and set $\Ucal  := \bigcup_{n=0}^\infty \Ocal_n$.
Clearly $\Ucal$ is $\sigma$-disjoint; it suffices to show that it is a $\pi$-base.
To see this, suppose that $V$ is a nonempty open subset of $K$.
Let $p$ be a nonempty regular open subset of $V$ and let
$\dot n$ be such that $p$ forces that $\check G_{\dot \xi_{\dot n}} \subseteq \check V$.
Now let $U$ be an element of $\Ucal$ which decides $\dot \xi_{\dot n}$ to be $\xi$.
Notice that we must have that $U \subseteq G_\xi \subseteq V$.
\end{proof}

Recall that a topological space $X$ is \emph{countably tight} if whenever $A \subseteq X$ and
$x \in \cl(A)$, there is a countable $A_0 \subseteq A$ such that $x \in \cl(A_0)$.
It is easy to show that continuous images of countably tight spaces are countably tight.
It is well known that in the class of compact Hausdorff spaces, countable tightness
is equivalent to the nonexistence of uncountable free sequences of regular pairs.
We now have the following corollary.

\begin{cor}\label{dense_metr}
Let $\Pcal$ be a forcing.
If $K$ is compact, contains a dense first countable subspace, and
\[
\one \forces_\Pcal \dot K \textrm{ is countably tight}
\]
then $K$ contains a dense metrizable subspace.
\end{cor}
 
\begin{proof}
By our assumption and Theorem \ref{sigma_disj_pi-base_char},
$K$ has a $\sigma$-disjoint $\pi$-base and thus so does the
dense first countable subspace.
By a result of H.E. White \cite{metrization:White},
any first countable Hausdorff space with a $\sigma$-disjoint
$\pi$-base has a dense metrizable subspace.
\end{proof}

An immediate consequence of the results we have developed so far is the following result
of Todorcevic.
Previously it had not been known whether there were nonseparable Rosenthal compacta which had no uncountable family of pairwise disjoint open sets
or whether certain specific Rosenthal compacta had dense metrizable subspaces
(see the discussion in \cite{Rosenthal_cpt}).
 
\begin{thm} \cite{Rosenthal_cpt}
Rosenthal compacta contain dense metrizable subspaces.
\end{thm}

\begin{proof}
Let $K$ be a Rosenthal compactum and let $\tilde K$ be the $\Qcal_K$-name for the reinterpretation
of $K$ as a Rosenthal compactum in the generic extension by $\Qcal_K$.
By Theorem \ref{RC_point_1st_ctbl}, $K$ contains a dense first countable subspace.
By Theorem \ref{RC_abs},
$\Qcal_K$ forces that $\tilde K$ is a Rosenthal compactum and hence is countably tight by
Theorem \ref{RC_ctbly_tight}.
Since $\dot K$, as defined in the beginning of this section, is a continuous projection
of $\tilde K$, it follows that $\dot K$ is forced to be countably tight.
By Corollary \ref{dense_metr}, $K$ contains a dense metrizable subspace.
\end{proof}

\section{Further reading}

As was mentioned earlier, Kunen's book \cite{set_theory:Kunen}
is a good next step if one is interested in further reading on forcing.
It also contains a large number of exercises.
Chapters VII and VIII provide a standard treatment of forcing, presented with a more semantic orientation, and
Chapter II provides some useful background on combinatorial set theory.

Further reading on forcings which add a single real --- such as $\Cohen$, $\Rand$, $\Mathias$, $\Ameoba_\epsilon$ ---
can be found in \cite{set_theory_reals}.
Also, Laver's work on the Borel Conjecture \cite{Borel_conj} is a significant early paper on the subject which already contains
important techniques in the modern set theory such as countable support iteration.
Zapletal's \cite{forcing_idealized} gives a different perspective on forcings related to set theory of the reals.

For those who can find a copy, \cite{forcing_appl} is also good further reading on forcing and provides a different perspective than \cite{set_theory:Kunen}.
Those readers who have studied the material on Martin's Axiom in \cite[II,VIII]{set_theory:Kunen}  
and/or \emph{forcing axioms} in \cite{forcing_appl} are referred to
\cite{PFA_ICM}, \cite{forcing_axioms:Todorcevic}, and \cite{FMS} where this concept is further developed
and the literature is surveyed.

Solovay's analysis of the model $L(\R)$ \cite{solovay_model}
is a landmark result in the study of forcing and large cardinals.
(Solovay actually analyzed a larger model than $L(\R)$, but $L(\R)$ has since shown itself
to be more fundamental and now bears the name \emph{Solovay model}.)
At the same time, \cite{solovay_model} should be accessible to readers who
have been through this article. 
A proof of Solovay's theorem is reproduced in \cite{higher_infinite} which is also
a standard encyclopedic reference on large cardinals. 
See also Mathias's infinite dimensional generalization
of Ramsey's Theorem which holds in $L(\R)$ after collapsing an appropriate large cardinal \cite{happy_families}.
An explanation of the special role Solovay's model plays in the foundations of mathematics
can be found in \cite{LCLM} \cite{L(R)_abs:Woodin}.
\cite{stationary_tower} provides a good introduction to the methods needed to establish absoluteness results
about $L(\R)$.


\def\Dbar{\leavevmode\lower.6ex\hbox to 0pt{\hskip-.23ex \accent"16\hss}D}

\end{document}